\documentclass{article}

\usepackage{amssymb,latexsym,amsmath,dsfont,mathrsfs}    
\usepackage{ stmaryrd }
\usepackage{enumerate}
\usepackage{amsthm} 
\usepackage{xspace}
\usepackage{comment}
\usepackage{ucs}
\usepackage[utf8x]{inputenc}

\newtheorem{theorem}{Theorem}[section]
\newtheorem{lemma}[theorem]{Lemma}
\newtheorem{prop}[theorem]{Proposition}
\newtheorem{defi}[theorem]{Definition}
\newtheorem{innercustomlemma}{Lemma}
\newenvironment{customlemma}[1]
  {\renewcommand\theinnercustomlemma{#1}\innercustomlemma}
  {\endinnercustomlemma}
\newtheorem{innercustomprop}{Proposition}
\newenvironment{customprop}[1]
  {\renewcommand\theinnercustomprop{#1}\innercustomprop}
  {\endinnercustomprop}
  \newtheorem{innercustomdefi}{Definition}
\newenvironment{customdefi}[1]
  {\renewcommand\theinnercustomdefi{#1}\innercustomdefi}
  {\endinnercustomdefi}
\theoremstyle{remark}
\newtheorem*{rem}{Remark}
\theoremstyle{remark}
\newtheorem*{ex}{Example}
\newcommand{\tendn}{\underset{n \rightarrow + \infty }{\longrightarrow}}
\newcommand{\Span}{\textnormal{Span}}
\newcommand{\graph}{\textnormal{graph}}

\begin{document}

\title{$p$-dimensional cones and applications}
\author{Maxence Novel}
\date{}
\maketitle

\begin{abstract}
We introduce a notion of $p$-dimensional cones made of $p$-dimensional subspaces and gauges on these cones, giving rise to a contraction principle which generalizes the one for Birkhoff cones. We prove a spectral gap result for the $p$ largest eigenvalues of an operator and a regularity result for the characteristic exponents of a random product of cone-contracting operators.
\end{abstract}

%%%%%%%%%%%%%%%%%%%%%%%%%%%%%%%%
%%%%%%%%%%%%%%%%%%%%%%%%%%%%%%%%

\section{Introduction}

The Perron-Frobenius Theorem \cite{Per07,Fro08} asserts that a real square matrix with strictly positive entries has a ‘spectral gap’, i.e. has a positive simple eigenvalue and the remaining part of the spectrum is strictly smaller in modulus.

In a seminal paper \cite{Bir57}, Birkhoff generalized this theory: he showed a contraction principle for projective cones $\mathcal{C}$ equipped with the Hilbert metric. If a linear map $T$ satisfies $T(\mathcal{C}^*) \subset \mathcal{C}^*$, then it is a contraction for this metric. Moreover, if the diameter of $T(\mathcal{C}^*)$ is finite for this metric, then $T$ is a strict contraction. 

These notions for real cones were extended to complex cones, introduced by Rugh \cite{Rug10}. The gauge used is no longer necessarily a metric, and rely on the hyperbolic metric on sections of the cones. Using this theory, he obtained a contraction principle for complex cones, a spectral gap theorem, and regularity for the first characteristic exponent of random products of operators preserving a complex cone. This was further studied by Dubois \cite{Dub09}, with a new gauge easier to use and which is a metric for some specific complex cones.

Can we use these theories to obtained some higher dimensional information ? For example, we would like to use cone contraction theory to obtain a 'p-dimensional' spectral gap for an operator $T$ : $T$ has $p$ largest eigenvalues $\lambda_1,\dots,\lambda_p$ (counted with multiplicity) with $|\lambda_1| \geq\dots\geq|\lambda_p|$, and the remaining part of the spectrum is contained in a disk centered at zero of radius strictly smaller than $|\lambda_p|$. As for characteristic exponents, we would like to get some regularity for the first $p$ exponents.

In \cite{Rue79}, Ruelle remarks that such analytic properties  can be obtained for some families of operators : he applied the result for the first exponent to the $p$-exterior product of operators $T\wedge\cdots\wedge T$. However, cone contraction can be difficult to check for $p$-exterior products, especially as there is no canonical norm on the exterior product $\bigwedge^pE$ of a Banach space $E$. More recently, Blumenthal, Quas and Gonzalez-Tokman have used several tools on grassmannian spaces to get more concise and intuitive proofs for the Oseledets theorem \cite{Blu16,GTQ15}. They also obtained, with Morris and Froyland, results on splittings and stability of characteristic exponents \cite{BM15,FGQ15}. This approach avoids the use of exterior algebra, making it more geometric and natural. In this paper, we will use similar tools but will still use the exterior algebra and link those two concepts to get our results.

In Section \ref{preliminaries}, we recall some results known on one-dimensional complex cones from \cite{Rug10} and \cite{Dub09}. We introduce the notion of complex cone, and some projective gauge $\delta_\mathcal{C}$, giving rise to a contraction principle. Under reasonable regularity assumptions on the complex cone, such a contraction has a spectral gap. In Section \ref{definitions}, we define the notion of $p$-dimensional cone as a union of $p$-dimensional subspaces which doesn't contain any $(p+1)$ dimensional subspace. We can construct a cone distance $d_{\mathcal{C}}(V,W)$ between $p$-dimensional subspace by taking the supremum of $\delta_\mathcal{C}(x,y)$ over $x \in V$ and $y \in W$. We then deduce from the contraction principle on one-dimensional complex cone a contraction principle for $p$-dimensional cones. In Section \ref{metrics}, we define two different metrics on the set $\mathcal{G}_p(E)$ of $p$-dimensional subspaces of $E$. The first is defined as the Hausdorff distance between the unit spheres and has some nice topological properties. The second one is obtained by identifying $V=\textnormal{Span}(x_1,\dots,x_p)$ with $x_1 \wedge \dots \wedge x_p \in \bigwedge\nolimits^pE$. We thus has to define and study a norm on $\bigwedge\nolimits^pE$, in order to get a very useful result : these two metrics on $\mathcal{G}_p(E)$ are equivalent. In Section \ref{spectral gap}, we generalize the regularity notions from complex cones to $p$-dimensional cones to obtain a spectral gap theorem for $p$-cone contractions. Finally, in Section \ref{random-section}, we prove a result on the analyticity of characteristic exponents of random products of cone contractions. As an example, if $(\xi_n)_{n \geq 0}$ is a sequence of independent and identically distributed random variables with values in $\mathbb{D}$, then for $t \in \mathbb{D}$ the characteristic exponents of the product of
$$M_n(t) = \begin{pmatrix}
10+t\xi_n &t+\xi_n & it \\
t+\xi_n & 6 &\xi_n \\
i\xi_n &0 &1
\end{pmatrix}$$
is analytic in $t$.

\section{Preliminaries}
\label{preliminaries}

In this section, we recall some definitions and results from \cite{Rug10} and \cite{Dub09}, which will serve as a foundation for our work. Let $E$ be a complex Banach space.

\begin{defi}
We say that a subset $\mathcal{C} \subset E$ is a closed complex cone if it is closed in $E$, $\mathbb{C}$-invariant and $\mathcal{C} \neq \{0\}$. Such a cone is said to be proper if it contains no complex planes. We will refer to proper closed complex cones as $\mathbb{C}$-cones.
\end{defi}

\begin{defi}
Given $\mathcal{C}$ a $\mathbb{C}$-cone and $x,y \in \mathcal{C}^*$. We define a gauge $\delta_{\mathcal{C}}:\mathcal{C}^*\times \mathcal{C}^* \rightarrow [0,\infty]$ between $x$ and $y$ as follows:
\begin{itemize}
\item If $x$ and $y$ are co-linear we set $\delta_{\mathcal{C}}(x,y)=0$
\item If $x$ and $y$ are linearly independant, we look at the section of $\mathcal{C}$ by the affine complex plane $x+\mathcal{C}(y-x)$. If $x$ and $y$ are in the same connected component $U$ of this section, we set $\delta_{\mathcal{C}}(x,y)=d_U(x,y)$ where $d_U$ is the hyperbolic metric on $U$. Else we set $\delta_{\mathcal{C}}(x,y)=\infty$.
\end{itemize}
This gauge is symmetric and projective, i.e. $\delta_{\mathcal{C}}(y,x)= \delta_{\mathcal{C}}(x,y) = \delta_{\mathcal{C}}(\lambda x,y) =\delta_{\mathcal{C}}(x,\lambda y)$.
\end{defi}

\begin{defi}
Given $\mathcal{C}$ a $\mathbb{C}$-cone and $x,y \in \mathcal{C}^*$. We consider the set $E_{\mathcal{C}}(x,y)=\{ z \in \mathbb{C} ~ |~ zx-y \notin \mathcal{C} \}$. We define a gauge $\delta_{\mathcal{C}}:\mathcal{C}^*\times \mathcal{C}^* \rightarrow [0,\infty]$ between $x$ and $y$. as follows:
\begin{itemize}
\item If $x$ and $y$ are co-linear we set $\delta_{\mathcal{C}}(x,y)=0$
\item If $x$ and $y$ are linearly independent, $\delta_{\mathcal{C}}(x,y)=\log(b/a)$, where $b=\sup |E_{\mathcal{C}}(x,y)|$ and $a= \inf |E_{\mathcal{C}}(x,y)|$.
\end{itemize}
This gauge is symmetric and projective.
\end{defi}

The following results are valid for each of these gauges.

\begin{prop}
\label{contraction-1}
Let $E_1, E_2$ be complex Banach spaces. Let  $T : E_1 \rightarrow E_2$ a continuous linear map and $\mathcal{C}_1 \subset E_1$,  $\mathcal{C}_2 \subset E_2$ two $\mathbb{C}$-cones such that $T(\mathcal{C}_1^*) \subset \mathcal{C}_2^*$. Then for all $x,y \in \mathcal{C}_1^*$, we have $\delta_{\mathcal{C}_2}(Tx,Ty) \leq \delta_{\mathcal{C}_1}(x,y).$

Moreover, if $\Delta = \textnormal{diam}_{\mathcal{C}_2}(T \mathcal{C}_1) < \infty$, there exists $\eta <1$ depending only on $\Delta$ such that for all $x,y \in \mathcal{C}_1^*$ we have $$\delta_{\mathcal{C}_2}(Tx,Ty) \leq \eta .\delta_{\mathcal{C}_1}(x,y).$$
\end{prop}

\begin{defi}
For $K>0$, we say that a $\mathbb{C}$-cone $\mathcal{C}$ is of $K$-bounded aperture if there exists a linear form $m:E \rightarrow \mathbb{C}$ such that $\forall x \in \mathcal{C}, \|m\|\|x\| \leq K |m(x)|$. \\The cone is of $K$-bounded sectional aperture if for all $x,y \in E$, $\mathbb{C}(x,y)$ is of $K$-bounded aperture, i.e. there exists a linear form $m_{x,y} : \textnormal{Span}(x,y) \rightarrow \mathbb{C}$ such that $\forall u \in \mathcal{C}(x,y), \|m\|\|u\| \leq K |m(u)|$.
\end{defi}

\begin{defi}
We say that a $\mathbb{C}$-cone $\mathcal{C}$ is reproducing if there exists a constant $A>0$ and $k \in \mathbb{N}^*$ such that for all $x \in E$, there exists $x_1, \dots, x_k \in \mathcal{C}$ satisfying $x=x_1+\dots+x_k$ and $\|x_1\|+\dots + \|x_k\| \leq A \|x\|$.
\end{defi}

\begin{theorem}
\label{gap-1}
Let $T \in \mathcal{L}(E)$ and let $\mathcal{C}$ be a reproducing $\mathbb{C}$-cone of bounded sectional aperture. Suppose that $T$ is a strict cone contraction, i.e. $T(\mathcal{C}^*) \subset \mathcal{C}^*$ and $\textnormal{diam}_{\mathcal{C}}(T \mathcal{C}) < \infty$ (cf Proposition \ref{contraction-1}). Then $T$ has a spectral gap.
\end{theorem}

\section{Multidimensional cones and contraction principle}
\label{definitions}

In \cite{Rug10}, H.H. Rugh defines a notion of proper complex cone which can be seen as an union of $\mathbb{C}$-lines which doesn't contain any $2$-dimensional subspace. We will first extend this definition to create higher dimensional cones. Then we will use the gauge/metric defined in \cite{Rug10} and \cite{Dub09} to define a "distance" on our cones and get a contraction principle similar to the one for proper complex cones.

In this section $E$ is a complex Banach space and $p$ is a positive integer. For a subset $X \subset E$, we denote $X^*=X \setminus \{0\}$.

\begin{defi}
We say that a subset $\mathcal{C} \subset E$ is a closed $p$-dimensional cone if $\mathcal{C}$ is a union of $p$-dimensional subspaces of $E$, is closed, and $\mathcal{C} \neq \emptyset$.

We say that a closed $p$ dimensional cone $\mathcal{C}$ is proper if it doesn't contains any $(p+1)$-dimensional subspace of $E$.   We will refer to proper closed $p$ dimensional cones as \emph{$p$-cones}.
\end{defi}

\begin{rem}
A $p$-cone is always $\mathbb{C}$-invariant. For $p=1$ we get the usual definition of proper closed complex cone.
\end{rem}

\begin{ex}
Let $F$ be a $p$-dimensional subspace of $E$, and $G$ a closed supplement of $F$ in $E$. Let $\pi:E \rightarrow F$ be the projection on $F$ parallel to $G$. Then for all $a>0$, the set $$\mathcal{C}_{\pi,a} = \{ x \in E ~ |~ \|(Id - \pi)(x)\| \leq a \|\pi(x)\| \}$$ is a $p$-cone of $E$.
\end{ex}

We now define a "distance" on the $p$-dimensional spaces of a $p$-cone, using the gauge defined on $1$-dimensional complex cones in \cite{Rug10} or in \cite{Dub09}. If $\mathscr{C}$ is a classic complex cone, we denote $\delta_\mathscr{C}$ this gauge on $\mathscr{C}$.

\begin{defi}
Let $\mathcal{C} \subset E$ be a $p$-cone and $x,y \in \mathcal{C}^*$. We define the 1-dimensional "distance" $d_{1,\mathcal{C}}(x,y)$ as $d_{1,\mathcal{C}}(x,y) = 0$ if Span($x$,$y$)$\subset \mathcal{C}$, and $d_{1,\mathcal{C}}(x,y) = \delta_{\mathcal{C} \cap Span(x,y)}(x,y)$ else (in this last case $\mathcal{C} \cap Span(x,y)$ is a proper complex cone). \\
Let $\mathcal{C}$ be a $p$-cone and $V,W \subset \mathcal{C}$ be $p$-dimensional subspaces of $E$. We define the "cone distance" between $V$ and $W$ as :
$$d_{\mathcal{C}}(V,W)= \underset{(x,y) \in (V^*,W^*)}{\sup} d_{1,\mathcal{C}}(x,y).$$
If  $\mathcal{C}'$ is a $p$-cone contained in $\mathcal{C}$, the diameter of $\mathcal{C}'$ in $\mathcal{C}$ is $\textnormal{diam}_{\mathcal{C}}(\mathcal{C}')=\sup \{ d_{\mathcal{C}}(V,W)~ |~ V,W \subset \mathcal{C}' ,p-\textnormal{dimensional} \}.$
\end{defi}

\begin{rem}
If Span($x$,$y$) $\nsubseteq \mathcal{C}$, then $d_{1,\mathcal{C}}(x,y) > 0$. \\
Let $V \neq W$ be $p$ dimensional spaces contained in $\mathcal{C}$. As $\mathcal{C}$ is proper, $V + W \nsubseteq \mathcal{C}$, so there exist $x \in V$, $y \in W$ such that Span($x$,$y$) $\nsubseteq \mathcal{C}$. This implies $d_{\mathcal{C}}(V,W) > 0$.
\end{rem}

\begin{ex}
Let $0<a<b$. Defining our "distance" with respect to the gauge for \cite{Dub09}, we have $\mathcal{C}_{\pi,a} \subset \mathcal{C}_{\pi,b}$ and $\textnormal{diam}_{\mathcal{C}_{\pi,b}}(\mathcal{C}_{\pi,a}) \leq 2 \log \left(\frac{b+a}{b-a}\right)$.
\end{ex}

\begin{proof}
Let $x,y \in \mathcal{C}_{\pi,a}$ with decompositions $x=x_F+x_G$ and $y=y_F+y_G$ with respect to $E=F \oplus G$. To estimate $\delta_{\mathcal{C}_{\pi,b}}(x,y)$ we must find an upper and lower bound for the set $E_{\mathcal{C}_{\pi,b}}(x,y)=\{ z \in \mathbb{c} ~ |~ zx-y \notin \mathcal{C}_{\pi,b} \}$. As the gauge is projective, we can assume $\|x_F\|=\|y_F\|=1$, and therefore we have $\|x_G	\|,\|y_G\| \leq a$. 

Let $z \in \mathbb{C}$ with $|z| \geq \frac{b+a}{b-a}$. Then 
$$\| z.x_G - y_G \| \leq |z|a+a = \frac{(b+a)a+(b-a)a}{b-a} = \frac{2ba}{b-a} = b ( \frac{b+a}{b-a} -1) \leq b (|z|-1) \leq b \| z.x_F - y_F \|,$$
so $zx-y \in \mathcal{C}_{\pi,b}$. Similarly, if $|z| \leq \frac{b-a}{b+a}$, then
$$\| z.x_G - y_G \| \leq |z|a+a = \frac{(b-a)a+(b+a)a}{b+a} = \frac{2ba}{b+a} = b (1- \frac{b-a}{b+a}) \leq b (1-|z|) \leq b \| z.x_F - y_F \|,$$
so $zx-y \in \mathcal{C}_{\pi,b}$. This gives $\inf |E_{\mathcal{C}_{\pi,b}}(x,y)| \geq \frac{b-a}{b+a}$ and $\sup |E_{\mathcal{C}_{\pi,b}}(x,y)| \leq \frac{b+a}{b-a}$, hence $$\delta_{\mathcal{C}_{\pi,b}}(x,y) \leq \log \left((\frac{b+a}{b-a})^2\right) = 2 \log \left(\frac{b+a}{b-a}\right).$$
\end{proof}

We can now formulate the following contraction principle:

\begin{prop}
\label{contraction}
Let $E_1, E_2$ be complex Banach spaces. Let  $T : E_1 \rightarrow E_2$ a continuous linear map and $\mathcal{C}_1 \subset E_1$,  $\mathcal{C}_2 \subset E_2$ two $p$-cones such that $T(\mathcal{C}_1^*) \subset \mathcal{C}_2^*$. Then for every $p$-dimensional spaces $V,W$ contained in $\mathcal{C}_1$, we have $d_{\mathcal{C}_2}(TV,TW) \leq d_{\mathcal{C}_1}(V,W).$

Moreover, if $\Delta = \textnormal{diam}_{\mathcal{C}_2}(T \mathcal{C}_1) < \infty$, there exists $\eta <1$ depending only on $\Delta$ such that for all $p$-dimensional spaces $V,W$ contained in $\mathcal{C}_1$, we have $$d_{\mathcal{C}_2}(TV,TW) \leq \eta .d_{\mathcal{C}_1}(V,W).$$
\end{prop}

\begin{proof}
The idea is to use the result known for the 1-dimensional cones. Given the definition of our "distance", the only points that matters are the ones for which Span($x$,$y$)$\nsubseteq \mathcal{C}$, and for these points Span($x$,$y$)$\cap \mathcal{C}$ is a proper closed 1-dimensional cone, allowing us to use the results of \cite{Rug10} and \cite{Dub09}. When Span($x$,$y$)$\nsubseteq \mathcal{C}$, we will denote $\mathcal{C}(x,y)=\textnormal{Span}(x,y)\cap \mathcal{C}$.

Let $V,W$ be $p$-dimensional subspaces of $E_1$ contained in $\mathcal{C}_1$. Let $x \in V^*$, $y \in W^*$. If Span($x$,$y$) $\subset \mathcal{C}_1$, then Span($Tx$,$Ty$) $=T$Span($x$,$y$) $\subset T \mathcal{C}_1 \subset \mathcal{C}_2$. As $T(\mathcal{C}_1^*) \subset \mathcal{C}_2^*$, we have $(TV)^*=T(V^*)$, $(TW)^*=T(W^*)$, and $TV$ and $TW$ are $p$-dimensional subspaces of $E_2$. The definition of $d_{\mathcal{C}_2}$ gives us:

\[ 
\begin{array}{ll}
d_{\mathcal{C}_2}(TV,TW) &= \sup \{ d_{1,\mathcal{C}_2}(x,y) ~ |~ (x,y) \in (TV^*,TW^*) \} \\
&= \sup \{ d_{1,\mathcal{C}_2}(Tx,Ty) ~ |~ (x,y) \in (V^*,W^*) \}\\
&= \sup \{ d_{1,\mathcal{C}_2}(Tx,Ty) ~ |~ (x,y) \in (V^*,W^*), \textnormal{Span($Tx$,$Ty$) $\nsubseteq \mathcal{C}_2$} \} \\
&= \sup \{ \delta_{\mathcal{C}_2(Tx,Ty)}(Tx,Ty) ~ |~ (x,y) \in (V^*,W^*), \textnormal{Span($Tx$,$Ty$) $\nsubseteq \mathcal{C}_2$} \}. \\

\end{array}
\]

For all $(x,y) \in (V^*,W^*)$ such that Span($Tx$,$Ty$) $\nsubseteq \mathcal{C}_2$, the map $T_{|\textnormal{Span($x$,$y$)}} : \textnormal{Span($x$,$y$)} \rightarrow \textnormal{Span($Tx$,$Ty$)}$ satisfies $T(\mathcal{C}_1(x,y)^*) \subset \mathcal{C}_2(Tx,Ty)^*$. Using Proposition \ref{contraction-1}, we get $\delta_{\mathcal{C}_2(Tx,Ty)}(Tx,Ty) \leq \delta_{\mathcal{C}_1(x,y)}(x,y)$. Hence
\[ 
\begin{array}{ll}
d_{\mathcal{C}_2}(TV,TW) &\leq \sup \{ \delta_{\mathcal{C}_1 (x,y)}(x,y) ~ |~ (x,y) \in (V^*,W^*), \textnormal{Span}(Tx,Ty) \nsubseteq \mathcal{C}_2\} \\
&\leq \sup \{ \delta_{\mathcal{C}_1(x,y)}(x,y) ~ |~ (x,y) \in (V^*,W^*), \textnormal{Span($x$,$y$) $\nsubseteq \mathcal{C}_1$} \} \\
&\leq \sup \{ d_{1,\mathcal{C}_1}(x,y) ~ |~ (x,y) \in (V^*,W^*), \textnormal{Span($x$,$y$) $\nsubseteq \mathcal{C}_1$} \} \\
&\leq \sup \{ d_{1,\mathcal{C}_1}(x,y) ~ |~ (x,y) \in (V^*,W^*) \} \\
&\leq d_{\mathcal{C}_1}(V,W).
\end{array}
\]

 For the second part of the statement, we show that $\textnormal{diam}_{\mathcal{C}_2(Tx,Ty)} (T(\mathcal{C}_1 (x,y))) \leq \Delta$ where the diameter is taken with respect to the gauge $\delta_{\mathcal{C}_2 (Tx,Ty)}$ on classic complex cones.\\
  Let $u,v \in \mathcal{C}_1(x,y)$. If $u$ and $v$ are colinear, $\delta_{\mathcal{C}_2 (Tx,Ty)}(Tu,Tv)=0$. Else, Span($u$,$v$) $=$ Span($x$,$y$), hence $\delta_{\mathcal{C}_2 (Tx,Ty)}(Tu,Tv) = \delta_{\mathcal{C}_2(Tu,Tv) }(Tu,Tv) \leq d_{1,\mathcal{C}_2}(Tu,Tv) $. As $Tu,Tv \in \mathcal{C}_1$, there exists $p$-dimensional subspaces $Z_u$ et $Z_v$ such that $u \in Z_u \subset \mathcal{C}_1$ (and similarly for $v$). Then $TZ_u \subset \mathcal{C}_2$ (and similarly for $v$) and $\delta_{\mathcal{C}_2(Tx,Ty)}(Tu,Tv) \leq d_{\mathcal{C}_2}(TZ_u,TZ_v) \leq \Delta$.
  
Using once again Lemma 2.4 in \cite{Rug10} or \cite{Dub09}, we get : for all $(x,y) \in (V^*,W^*)$ such that Span($Tx$,$Ty$) $\nsubseteq \mathcal{C}_2$, $\delta_{\mathcal{C}_2 (Tx,Ty)}(Tx,Ty) \leq \eta.\delta_{\mathcal{C}_1(x,y)}(x,y)$, and the same calculation as before gives the desired result.

\end{proof}

\begin{ex}
For $0<a<b$, if $T \in \mathcal{L}(E)$ satisfies $T(\mathcal{C}_{\pi,b}^*) \subset \mathcal{C}_{\pi,a}^*$, then $T$ satisfies the assumptions of Proposition \ref{contraction} and we have (in the case we use the metric from \cite{Dub09}):
For all $p$-dimensional spaces $V,W$ contained in $\mathcal{C}_{\pi,b}$, $$d_{\mathcal{C}_{\pi,b}}(TV,TW) \leq \frac{a}{b} .d_{\mathcal{C}_{\pi,b}}(V,W),$$
i.e. $T$ is a contraction on the set of $p$-dimensional spaces of  $\mathcal{C}_{\pi,b}$ with respect to the "distance" $d_{\mathcal{C}_{\pi,b}}$.
\end{ex}

\section{Metrics on the Grassmannian}
\label{metrics}

In this section, we define metrics on the space $\mathcal{G}(E)$ of closed subspaces of $E$. First a Hausdorff metric, with interesting topological properties, and then an "exterior algebra" metric defined by seeing a $p$-dimensional subspace as a vector line in $\bigwedge\nolimits^pE$. We show that these metrics are equivalent and obtain some interesting properties.

\begin{defi}
We define on $\mathcal{G}(E)$ the metric $d_H$ given by the Hausdorff distance between the unit spheres :
$$\forall V,W \in \mathcal{G}(E), d_H(V,W)=d_H(S_V,S_W)=\max \left( \sup\limits_{x \in S_V} d(x,S_W),\sup\limits_{y \in S_W} d(y,S_V)\right).$$
\end{defi}

The metric space $\mathcal{G}(E),d_H)$ satisfies the following properties:

\begin{prop}
\begin{enumerate}[(i)]
\item $\mathcal{G}(E),d_H)$ is complete.
\item The subset $\mathcal{G}_p(E)$ of all $p$-dimensional subspaces of $E$ is closed in $\mathcal{G}(E),d_H)$.
\item Defining $\delta(V,W)=\max \left( \sup\limits_{x \in S_V} d(x,W),\sup\limits_{y \in S_W} d(y,V)\right)$, we have:
$$\forall V,W \in \mathcal{G}(E), \delta(V,W) \leq d_H(V,W) \leq 2 \delta(V,W).$$
In particular, the metrics $d_H$ and $\delta$ are equivalent on $\mathcal{G}(E)$.
\end{enumerate}
\end{prop}

Before defining an "exterior algebra" metric on $\mathcal{G}_p(E)$, we must define a norm on $\bigwedge\nolimits^p E$. Most of the proofs relies on computations on $\bigwedge\nolimits^p E$. Although they are not particularly difficult, they aren't particularly interesting for a first reading. Those proofs are compiled in the appendix.

\begin{defi}
\label{defi1}
If $l_1,\cdots,l_p \in E'$ and $x_1,\cdots,x_p \in E$, we define $$\langle l_1 \wedge \cdots \wedge l_p ,x_1 \wedge \cdots \wedge x_p \rangle = \det ((\langle l_i, x_j\rangle)_{i,j})$$
which can be extended to a linear form $l_1 \wedge \cdots \wedge l_p$ on $\bigwedge\nolimits^p E$. We define a norm $||.||_{\wedge 1, p}$ on $\bigwedge\nolimits^p E$ by

\[ ||x||_{\wedge1, p} = \underset{\begin{array}{c}l_1,\cdots,l_p \in E' \\ ||l_1||=\cdots =||l_p||=1 \end{array}}{\sup} \left| \langle l_1 \wedge \cdots \wedge l_p ,x \rangle \right| \]
\end{defi}

\begin{defi}
\label{defi2}
We define a norm $||.||_{\wedge 2, p}$ on $\bigwedge\nolimits^p E$ by
\item \[ \|x\|_{\wedge2,p} = \inf \underset{i}{\sum} \|x^i_1\| \cdots \|x^i_p \| \]
where the infimum is taken over all decompositions $x = \underset{i}{\sum} x^i_1 \wedge \cdots \wedge x^i_p $.
\end{defi}

\begin{prop}
\label{prop1}
On $\bigwedge\nolimits^p E$, $\|.\|_{\wedge1,p} \leq (\sqrt{p})^p \|.\|_{\wedge2,p}$.
\end{prop}

Let's now look at some properties of these norms.

\begin{lemma} 
\label{p->p+1}
Let $x \in E$ and $u \in \bigwedge\nolimits^pE$. Then
\begin{enumerate}[(i)]
\item $|| x \wedge u ||_{\wedge1,p+1} \leq (p+1) ||x||.||u||_{\wedge1,p}$
\item $|| x \wedge u ||_{\wedge2,p+1} \leq ||x||.||u||_{\wedge2,p}$
\end{enumerate}
\end{lemma}

\begin{lemma}
\label{norme-difference}
Let $x_1, \cdots , x_p , y_1 ,\cdots , y_p \in E$ such that $\forall i, ||x_i|| = ||y_i|| = 1$. Then 
\begin{enumerate}[(i)]
\item $||x_1 \wedge \cdots \wedge x_p - y_1 \wedge \cdots \wedge y_p ||_{\wedge1,p} \leq p.p!.\max (||x_i - y_i||)$,
\item $||x_1 \wedge \cdots \wedge x_p - y_1 \wedge \cdots \wedge y_p ||_{\wedge2,p} \leq p.\max (||x_i - y_i||)$.
\end{enumerate}
\end{lemma}

One problem we can have is that $\|x_i\|=1$ but with $\|x_1 \wedge \cdots \wedge x_p\|_{\wedge,p}$ being small. To prevent this from happening, we define a notion which resemble orthogonality.

\begin{defi}
Let $E$ be a normed vector space and $x \in E$. We say that a linear form $l: E \rightarrow \mathds{C}$ is adapted to $x$ if $||l||=1$ and $l(x) = 1$, which is equivalent to $||l||=1$ and $\forall y \in \ker l, ||x + y|| \geq ||x||$. The hyperplane $\ker l$ is said to be adapted to $x$.

Such a linear form always exists (consequence of the Hahn-Banach theorem).
\end{defi}

\begin{rem}
The easiest way to think about this is to see $\ker l$ as a set of vectors "orthogonal" to $x$. The difference with the euclidian case is that given such a $z$ "orthogonal" to $x$, $x$ is not necessarily "orthogonal" to $z$.
\end{rem}

%FIGURE

\begin{defi}
Let $E$ be a normed vector space and $V$ be a $p$-dimensional subspace of $E$. We say that $(x_1,\cdots,x_p)$ is a right decomposition of $V$ if: $\forall i,||x_i||=1$ and there exists linear forms $l_1,\cdots,l_p$ such that $\forall i$, $l_i$ is adapted to $x_i$ and $\forall j > i, x_j \in \ker l_i$.
We say that the $l_i$ are linear forms adapted with the decomposition $(x_1,\cdots,x_p)$.
\end{defi}

\begin{rem}
This definition means that $x_2, \cdots, x_p$ are "orthogonal" to $x_1$ ; $x_3, \cdots, x_p$ are "orthogonal" to $x_2$, etc. This would be our equivalent of an orthogonal basis.
\end{rem}

\begin{lemma}
\label{base-norme}
If $(x_1,\cdots,x_p)$ is a right decomposition of $V$ and $(l_1,\cdots,l_p)$ are adapted linear forms, then $(x_1,\cdots,x_p)$ is a basis of $V$ and the $l_i$ are independant (and therefore are forming a basis of $V'$). Moreover, $||x_1 \wedge \cdots \wedge x_p||_{\wedge1,p} \geq 1$.
\end{lemma}

\begin{proof}
\[  \begin{array}{llr}
\langle l_1 \wedge \cdots \wedge l_{p} , x_1 \wedge \cdots \wedge x_p \rangle &= \det (\langle l_i , x_j \rangle) \\
&= \begin{vmatrix} 1 & 0 & \cdots & 0 \\ & 1 & \ddots & \vdots\\ & (*) & \ddots & 0 \\ & & & 1 \end{vmatrix} & \textnormal{as } \forall i, \langle l_i,x_i \rangle =1 \textnormal{ and } \forall j>i, x_j \in \ker l_i\\
&= 1
\end{array}\]

So the $l_i$ are independent, the $x_i$ are independent and $||x_1 \wedge \cdots \wedge x_p||_{\wedge1,p} \geq 1$.
\end{proof}

\begin{rem}
Using the Hahn-Banach theorem, if  $(x_1,\cdots,x_p)$ is a right decomposition of $V$ in $V$,  $(x_1,\cdots,x_p)$ is a right decomposition of $V$ in $E$, i.e. being a right decomposition only depends on the $x_i$.
\end{rem}

\begin{prop} \emph{(completion of a right decomposition).}
Let $E$ be a normed vector space and $V$ be a $p$-dimensional subspace of $E$. Given $x_1, \cdots, x_k \in V$ such that $(x_1, \cdots, x_k)$ is a right decomposition of $Vect(x_1, \cdots, x_k)$, there exists $x_{k+1},\cdots,x_p \in V$ s.t. $(x_1, \cdots, x_p)$ is a right decomposition of $V$. In particular, $V$ has a right decomposition.
\end{prop}

\begin{proof}
Let $k < p$ and let $l_1,\cdots,l_k$ be linear forms adapted to the decomposition $(x_1, \cdots, x_k)$. As $\dim(\underset{i=1}{\overset{k}{\cap}} \ker l_i) \geq m-k >0$, let's take $x_{k+1} \in \underset{i=1}{\overset{k}{\cap}} \ker l_i$ and let $l_{k+1}$ be a linear form adapted to $x_{k+1}$. Then $(x_1, \cdots, x_{k+1})$ is a right decomposition of $Vect(x_1, \cdots, x_{k+1})$. By iteration, we get: $(x_1, \cdots, x_p)$ is a right decomposition of $Vect(x_1, \cdots, x_p)$, and given Lemma \ref{base-norme} and $\dim(V)=p$, $(x_1, \cdots, x_p)$ is a right decomposition of $V$.
\end{proof}

\begin{lemma}
\label{estimation-coordonnees}
Let $(x_1,\cdots,x_p)$ be a right decomposition of $V$ and $u = \underset{i=1}{\overset{p}{\sum}} a_i x_i \in V$. Then $\forall i, |a_i| \leq 2^{i-1} ||u||$.
\end{lemma}

\begin{proof}
As $\langle l_i , u \rangle=\underset{k=1}{\overset{i-1}{\sum}} a_k \langle l_i , x_k \rangle + a_i$ and $||l_i||=1$, we get $|a_i| \leq ||u||+ \underset{k=1}{\overset{i-1}{\sum}} |a_k|$ and therefore $|a_i| \leq 2^{i-1} ||u||$.
\end{proof}

Now we have all the tools to define a new metric on $\mathcal{G}_p(E)$ and compare it to the metric $d_H$. For that, see that if $(x_1,\dots,x_p)$ and $(y_1,\dots,y_p)$ are bases of a $p$-dimensional space $V$ of $E$, then $x_1 \wedge \dots \wedge x_p$ and $y_1 \wedge \dots \wedge y_p$ are colinear in $\bigwedge\nolimits^pE$. We can thus identify the $p$-dimensional space $V$ of $E$ to the vector line $\textnormal{Span}(x_1 \wedge \dots \wedge x_p)$ of $\bigwedge\nolimits^p E$.

The following definition and theorem are proved for the norm $\|.\|_{\wedge,p}=\|.\|_{\wedge1,p}$. However, they remain true for $\|.\|_{\wedge,p}=\|.\|_{\wedge2,p}$, thanks to Lemma \ref{equiv exterior}.

\begin{defi}
For $V \in  \mathcal{G}_p(E)$, $\bigwedge\nolimits^pV=V \wedge \cdots \wedge V$ is a 1-dimensional subspace of $\bigwedge\nolimits^p E$. Hence we define for $V, W \in \mathcal{G}_p(E)$ :
\[ d_{\wedge,p}(V,W)= \inf \{ ||v - w||_{\wedge,p} ~ |~  (v,w) \in \bigwedge\nolimits^pV \times \bigwedge\nolimits^p W \textnormal{ with }||v||_{\wedge,p}=||w||_{\wedge,p}=1 \} \]

This corresponds to the grassmannian metric between the complex lines $\bigwedge\nolimits^pV$ and $\bigwedge\nolimits^pW$ in $(\bigwedge\nolimits^pE, ||.||_{\wedge,p})$.
\end{defi}

\begin{rem}
If $x \in \bigwedge\nolimits^pV, y \in \bigwedge\nolimits^pW$ with $\|x\|_{\wedge,p}=\|y\|_{\wedge,p}=1$, then $d_{\wedge,p}(V,W)=\inf\limits_{|\lambda|=1}\|x-\lambda y\|_{\wedge,p}$.
\end{rem}

\begin{theorem}
\label{metriques-equivalentes}
For $p \in \mathbb{N}$, there exists constants $c_p,c_p' >0$ such that for all $E$ normed vector space, for all $V,W \in \mathcal{G}_p(E)$, 
\[ c_p d_{H}(V,W) \leq  d_{\wedge,p}(V,W) \leq c_p' d_H(V,W) \]
The metrics $d_H$ and $d_{\wedge,p}$ are equivalent on $\mathcal{G}_p(E)$.
\end{theorem}

 \begin{proof}
Let's first prove the second inequality. Let $p \in \mathbb{N}$, $E$ be a normed vector space, $V,W \in \mathcal{G}_p(E)$ and $d=d_H(V,W)$. Let $ x_1 \wedge \cdots \wedge x_p$ be a right decomposition of $V$. For all $i$, there exists $y_i \in W$ such that $||x_i-y_i|| \leq d$ and  $||y_i|| = 1$. Then, by Lemma \ref{norme-difference}, $ ||x_1 \wedge \cdots \wedge x_p - y_1 \wedge \cdots \wedge y_p || \leq d.p.p!$.
As $ x_1 \wedge \cdots \wedge x_p$ is a right decomposition of $V$, $||x_1 \wedge \cdots \wedge x_p || \geq 1$, so for $x= \frac{x_1 \wedge \cdots \wedge x_p }{||x_1 \wedge \cdots \wedge x_p ||} \in \bigwedge\nolimits^pV$ and $y = \frac{y_1 \wedge \cdots \wedge y_p }{||x_1 \wedge \cdots \wedge x_p ||} \in \bigwedge\nolimits^mW$, we have $||x||=1$ and $||x -y|| \leq  d.p.p!$.Thus $dist(x,\textnormal{Span}(y)) \leq  d.p.p!$ and $d_{\wedge,p}(V,W) \leq (2.p.p!).d$. The constant $c_p' = 2.p.p!$ works.  \\ ~ \\ 

Now let's prove recursively the remaining inequality of our proposition. Let $(H_p)$ : "there exists $c_p >0$ such that for all $E$ normed vector space, for all $V,W \in \mathcal{G}_1(E)$, $c_p d_H(V,W) \leq  d_{\wedge,p}(V,W)$". $(H_1)$ is true as $d_H(V,W) =  d_{\wedge,1}(V,W)$ for $V,W \in \mathcal{G}_p(E)$.
Let's assume $(H_p)$. Let $E$ be a normed vector space and let $V,W \in Gr_{p+1}(E)$. Let $d=d_H(V,W)$. We divide the proof in 3 cases.

\begin{lemma} Assume there exists $V'$ and $W'$ subspaces of dimension $p$ of (respectively) $V$ and $W$ such that $d_H(V',W') \leq \frac{1}{2^{p+5}(p+1)(p+1)!^2} d$ (which means: $V$ and $W$ are very very close on $p$ directions, and the 'last direction' is what gives the Hausdorff distance $d$ between $V$ and $W$). Then $d_{\wedge,p+1}(V,W) \geq \frac{1}{(p+1)!2^{p+3}}d$.
\end{lemma}

\begin{proof} The idea is to complete a right decomposition of $V'$ and $W'$ to get a right decomposition of $V$ and $W$. Then, when we estimate $\|x - \lambda y\|_{\wedge,p+1}$, we get two terms : one corresponding to the $p$ first directions (small term compared to $d$), and one corresponding to the last direction (term with size comparable to $d$).

Let $x_1, \cdots, x_p$ be a right decomposition of $V'$ and $H_1, \cdots, H_p$ the adapted hyperplanes. For all $i \in [1..p]$, there exists $y_i \in W'$ such that $||y_i|| = 1$ and $||x_i - y_i|| \leq  d_H(V',W') \leq \frac{1}{2^{p+5}(p+1)!^2} d$. Let $x_{p+1} \in V\cap H_1 \cap \cdots \cap H_p$ and $y_{p+1} \in W\cap H_1 \cap \cdots \cap H_p$ with $||x_{p+1}||=||y_{p+1}|| = 1$  ($(x_1, \cdots,x_{p+1})$ and $(x_1, \cdots, x_p,y_{p+1})$ are right decompositions) and let $\lambda \in \mathds{C}^*$ such that $|\lambda| = 1$. Denoting $x=\dfrac{x_1 \wedge \cdots \wedge x_{p+1}}{||x_1 \wedge \cdots \wedge x_{p+1}||}$, $y=\dfrac{y_1 \wedge \cdots \wedge y_{p+1}}{||y_1 \wedge \cdots \wedge y_{p+1}||}$ and $k= \dfrac{||x_1 \wedge \cdots \wedge x_{p+1}||}{||y_1 \wedge \cdots \wedge y_{p+1} ||}$, we have:

\[ \begin{array}{rlr}
(p+1)!||x - \lambda y|| &=\frac{(p+1)!}{||x_1 \wedge \cdots \wedge x_{p+1}||}.||  x_1 \wedge \cdots \wedge x_{p+1} -  \lambda k y_1 \wedge \cdots \wedge y_{p+1} || \\
&\geq ||  x_1 \wedge \cdots \wedge x_{p+1} -  \lambda k y_1 \wedge \cdots \wedge y_{p+1} || \\
&\geq||  x_1 \wedge \cdots \wedge x_p \wedge (x_{p+1} -  \lambda k y_{p+1}) + (x_1 \wedge \cdots \wedge x_p - y_1 \wedge \cdots \wedge y_p) \wedge \lambda k y_{p+1} || \\
&\geq \underbrace{||  x_1 \wedge \cdots \wedge x_p \wedge (x_{p+1} -  \lambda k y_{p+1}) ||}_{A_1} - \underbrace{|| (x_1 \wedge \cdots \wedge x_p - y_1 \wedge \cdots \wedge y_p) \wedge \lambda k y_{p+1} ||}_{A_2} \\
\end{array}
\]

\textbf{Estimation of $A_2= || (x_1 \wedge \cdots \wedge x_p - y_1 \wedge \cdots \wedge y_p) \wedge \lambda k y_{p+1} ||$}.\\ As $||y_{p+1}|| = 1$, $A_2 \leq (p+1)k.|| x_1 \wedge \cdots \wedge x_p - y_1 \wedge \cdots \wedge y_p ||$. Given how where  chosen the $y_i$ and looking at the very first part of the proof, we get \[ A_2 \leq (p+1)k.(2p.p!.d_H(V',W')) \leq \frac{k}{2^{p+5}(p+1)!}.d \]Let's find an upper bound for $k$. We have $||x_1 \wedge \cdots \wedge x_{p+1}|| \leq (p+1)!$. Moreover, 
\[ \begin{array}{llr}||y_1 \wedge \cdots \wedge y_{p+1} || &\geq ||x_1 \wedge \cdots \wedge x_{p} \wedge y_{p+1} || - || (x_1 \wedge \cdots \wedge x_{p} - y_1 \wedge \cdots \wedge y_{p}) \wedge y_{p+1}||  \\
&\geq 1 - (p+1) ||x_1 \wedge \cdots \wedge x_{p} - y_1 \wedge \cdots \wedge y_{p}||\\
&\geq 1 - (p+1)^2.(p+1)!.d_H(V',W')\\
&\geq 1 - \frac{1}{2^{p+5}}.d \geq \frac{1}{2} 
\end{array}\]
Hence $k \leq 2(p+1)!$ and $A_2 \leq \frac{1}{2^{p+4}}.d$.

\textbf{Estimation of $A_1=||  x_1 \wedge \cdots \wedge x_p \wedge (x_{p+1} -  \lambda k y_{p+1}) ||$}.\\ Let $u= \underset{i=1}{\overset{p+1}{\sum}}a_i x_i \in V$ with $||u||=1$ and $d(u, S_W) \geq \frac{d}{2}$ (such a $u$ exists as $d_H(V,W)=d$). Then $d(u, W) \geq \frac{d}{4}$ and so
\[ \frac{d}{4} \leq ||\underset{i=1}{\overset{p+1}{\sum}}a_i x_i - (\underset{i=1}{\overset{p}{\sum}}a_i y_i + a_{p+1} \lambda k y_{p+1})|| \leq \underset{i=1}{\overset{p}{\sum}}|a_i|.|| x_i- y_i|| + |a_{p+1}|.||x_{p+1}- \lambda k y_{p+1}|| \]
As $x_1,\cdots,x_{p+1}$ is a right decomposition of  $V$ and $||u||=1$, Lemma \ref{estimation-coordonnees} gives $|a_i| \leq 2^{i-1}$. Then
\[ 2^p ||x_{p+1}- \lambda k y_{p+1}|| \geq \frac{d}{4} - \underset{i=1}{\overset{p}{\sum}} 2^{i-1}||x_i-y_i|| \geq \frac{d}{4} - 2^pd_H(V',W') \geq \frac{d}{4} - \frac{d}{2^4}.\]

Given how we've chosen $x_{p+1}$ and $y_{p+1}$, $(x_1, \cdots, x_p, \frac{x_{p+1}- \lambda k y_{p+1}}{||x_{p+1}- \lambda k y_{p+1}||}) $ is a right decomposition. Using Lemma \ref{base-norme}, we have $A_1 \geq ||x_{p+1}- \lambda k y_{p+1}|| \geq (\frac{1}{2^{p+2}}-\frac{1}{2^{p+4}}).d$. Hence
\[||x - \lambda y||  \geq \frac{1}{(p+1)!} (A_1 - A_2) \geq \frac{1}{(p+1)!} (\frac{1}{2^{p+2}}-\frac{1}{2^{p+4}} -\frac{1}{2^{p+4}} ).d \geq \frac{1}{(p+1)!2^{p+3}}.d .\]
which holds for all $\lambda$ with $|\lambda|=1$, giving the desired $d_{\wedge,p+1}(V,W) \geq \frac{1}{(p+1)!2^{p+3}}d$.
\end{proof}

\begin{lemma}Assume that for all $V',W'$ subspaces of dimension $m$ of (respectively) $V$ and $W$, we have $d_H(V',W') > k_p d$ where $k_p=\frac{1}{2^{p+5}(p+1)^2(p+1)!}$. Moreover, assume that there exists $x_{p+1} \in V, y_{p+1} \in W$ with $||x_{p+1}||=||y_{p+1}||=1$ such that $||x_{p+1} - y_{p+1} || \leq  \frac{c_p.k_p}{2(p+1)!}d$ (which means: $V$ and $W$ are very very close on one specific direction). Then $d_{\wedge,p+1}(V,W) \geq \frac{c_p.k_p}{2(p+1)!}d$.
\end{lemma}

\begin{proof}
Let $H$ be a hyperplane adapted to $y_{p+1}$. Let $(x_1, \cdots, x_p)$ be a  right decomposition of $V \cap H$ and $(y_1, \cdots, y_p)$ be a right decomposition of $W \cap H$ (in particular $(y_{p+1},y_1, \cdots, y_p)$ is a right decomposition of $W$). As $||x_{p+1} - y_{p+1} || < 1$, $x_{p+1} \notin H$ so $\dim(V \cap H) = p =\dim(W \cap H)$. As in case 1, let $\lambda \in \mathds{C}$ with $|\lambda|=1$ and let $x=\dfrac{x_1 \wedge \cdots \wedge x_{p+1}}{||x_1 \wedge \cdots \wedge x_{p+1}||}$, $y=\dfrac{y_1 \wedge \cdots \wedge y_{p+1}}{||y_1 \wedge \cdots \wedge y_{p+1}||}$ and $k= \dfrac{||x_1 \wedge \cdots \wedge x_{p+1}||}{||y_1 \wedge \cdots \wedge y_{p+1} ||}$. Then

\[ \begin{array}{rlr}
(p+1)!||x - \lambda y|| &=\frac{(p+1)!}{||x_1 \wedge \cdots \wedge x_{p+1}||}.||  x_1 \wedge \cdots \wedge x_{p+1} -  \lambda k y_1 \wedge \cdots \wedge y_{p+1} || \\
&\geq ||  x_1 \wedge \cdots \wedge x_{p+1} -  \lambda k y_1 \wedge \cdots \wedge y_{p+1} || \\
&\geq||  x_1 \wedge \cdots \wedge x_p \wedge (x_{p+1} -  y_{p+1}) + (x_1 \wedge \cdots \wedge x_p - \lambda k y_1 \wedge \cdots \wedge y_p) \wedge y_{p+1} || \\
&\geq \underbrace{|| (x_1 \wedge \cdots \wedge x_p - \lambda k y_1 \wedge \cdots \wedge y_p) \wedge y_{p+1} ||}_{B_1} - \underbrace{||  x_1 \wedge \cdots \wedge x_p \wedge (x_{p+1} - y_{p+1}) ||}_{B_2}\\
\end{array}
\]

Let's estimate $B_2$. We have $B_2 \leq (p+1)! ||x_{p+1} - y_{p+1}|| \leq \frac{c_p.k_p}{2}.d$.

Let's estimate $B_1$. 
\[ \begin{array}{llr}
B_1 &\geq d(x_1 \wedge \cdots \wedge x_p, \textnormal{Span}(y_1 \wedge \cdots \wedge y_p)) \\
&\geq d(\frac{x_1 \wedge \cdots \wedge x_p}{||x_1 \wedge \cdots \wedge x_p||}, \textnormal{Span}(y_1 \wedge \cdots \wedge y_p)) &\textnormal{as } ||x_1 \wedge \cdots \wedge x_p|| \geq 1 \\
&\geq d_{\wedge,p}( V \cap H, W \cap H) \\
&\geq c_p d_H( V \cap H, W \cap H) &\textnormal{using our recursion hypothesis $(H_p)$} \\
&\geq c_p.k_p.d &\textnormal{by the first assumption of this Lemma}
\end{array}
\]

Thus we got 
$ ||x - \lambda y|| \geq \frac{1}{(p+1)!} (B_1 - B_2) \geq \frac{c_p.k_p}{2(p+1)!}d$ for all $\lambda$ with $|\lambda|=1$, hence $d_{\wedge,p+1}(V,W) \geq \frac{c_p.k_p}{2(p+1)!}d$.
\end{proof}

\begin{lemma}Assume that for all $V',W'$ subspaces of dimension $p$ of (respectively) $V$ and $W$, we have $d_H(V',W') > k_p d$ where $k_p=\frac{1}{2^{p+5}(p+1)^2(p+1)!}$. Moreover, assume that for all $x \in V, y \in W$ with $||x||=||y||=1$, $||x - y|| >  \frac{c_p.k_p}{2(p+1)!}d$ (which means: $V$ and $W$ are a little far away from each other in every direction). Then $d_{\wedge,p+1}(V,W) \geq \frac{c_p.k_p}{4(p+1)!}d$.
\end{lemma}

\begin{proof}
In this case, $V \cap W = \{ 0 \}$. Let $(x_1, \cdots, x_{p+1})$ be a right decomposition of $V$ and let $(l_1',l_2, \cdots, l_{p+1})$ be the adapted linear forms (with norm $1$). Let $(y_1, \cdots, y_{p+1})$ be a basis of $W$. Let $\lambda \in \mathds{C}$. We define 
\[ \begin{array}{rccc}
l_1 : &\textnormal{Vect}(V,W) &\longrightarrow &\mathds{C} \\
&\underset{i=1}{\overset{p+1}{\sum}}a_i x_i + \underset{i=1}{\overset{p+1}{\sum}}b_i y_i &\longmapsto &a_1 
\end{array}
\]

Denoting $x = x_1 \wedge \cdots \wedge  x_{p+1}$ and $y =  y_1 \wedge \cdots \wedge  y_{p+1}$, we get

\[  
\langle l_1 \wedge \cdots \wedge l_{p+1} , x - \lambda y \rangle = \det (\langle l_i , x_j \rangle) - \underbrace{\det (\langle l_i , y_j \rangle)}_{=0 \textnormal{ as }\langle l_1 , y_j \rangle =0~  \forall j} = 1 \textnormal{~ ~ ~ ~ ~ (as in Lemma \ref{base-norme})}
\]
Let's estimate $||l_1||$. Let $u = \underset{i=1}{\overset{p+1}{\sum}}a_i x_i + \underset{i=1}{\overset{p+1}{\sum}}b_i y_i$.
\[ \begin{array}{rlr}
||u|| &\geq d(\underset{i=1}{\overset{p+1}{\sum}}a_i x_i, W) \\
&\geq || \underset{i=1}{\overset{p+1}{\sum}}a_i x_i|| .d\left(\frac{\underset{i=1}{\overset{p+1}{\sum}}a_i x_i}{ ||\underset{i=1}{\overset{p+1}{\sum}}a_i x_i||}, W\right) \\
&\geq || \underset{i=1}{\overset{p+1}{\sum}}a_i x_i||.\frac{1}{2} .d\left(\frac{\underset{i=1}{\overset{p+1}{\sum}}a_i x_i}{|| \underset{i=1}{\overset{p+1}{\sum}}a_i x_i||}, S_W\right) \\
&\geq || \underset{i=1}{\overset{p+1}{\sum}}a_i x_i||.\frac{1}{2}.\frac{c_p.k_p}{2(p+1)!}.d & \textnormal{given the second assumption of this lemma} \\
&\geq \frac{c_p.k_p}{4(p+1)!}.d.|a_1| &\textnormal{as $(x_1, \cdots, x_{p+1})$ is a right decomposition of $V$}
\end{array}
\]
Hence $||\frac{c_p.k_p}{4(p+1)!}d. l_1|| \leq 1$ and as $\langle \frac{c_p.k_p}{4(p+1)!}d. l_1 \wedge \cdots \wedge l_{p+1} , x - \lambda y \rangle = \frac{c_p.k_p}{4(p+1)!}d$ we get $||x - \lambda y|| \geq \frac{c_p.k_p}{4(p+1)!}d$.
As $||x|| \geq 1$, $d(\frac{x}{||x||},\textnormal{Span}(y)) \geq d(x,\textnormal{Span}(y))$ and therefore $d_{\wedge,p+1}(V,W) \geq \frac{c_p.k_p}{4(p+1)!}d$. 
\end{proof}

Combining the three previous lemmas and taking $c_{p+1} = \min(\frac{1}{(p+1)!2^{p+3}},\frac{c_p.k_p}{2(p+1)!},\frac{c_p.k_p}{4(p+1)!})$, we have $d_{\wedge,p+1}(V,W) \geq c_{p+1}d_H(V,W)$, achieving the proof of Theorem \ref{metriques-equivalentes}.

\end{proof}

\begin{lemma}
\label{equiv exterior}
Let $k \in \mathbb{N}$. The norms $\|.\|_{\wedge1,p}$ and $\|.\|_{\wedge2,p}$ are equivalent on the subset $S_k$ of $\bigwedge\nolimits^pE$ containing the sums of at most $k$ decomposable tensors.
\end{lemma}

\begin{proof}
From Proposition \ref{prop1}, on $\bigwedge\nolimits^p E$, $\|.\|_{\wedge1,p} \leq (\sqrt{p})^p \|.\|_{\wedge2,p}$.

Let $x=\underset{i=1}{\overset{k}{\sum}} x_1^i \wedge \cdots \wedge x_p^i \in S_k$. Let $F=\Span(x_k^i)$. Then $\dim F \leq kp$. We consider $(e_1,\dots,e_d)$ an adapted decomposition of $F$ and $(l_1,\dots,l_d)$ the corresponding linear forms ($d \leq kp$). In this basis, $x=\underset{i_1<\dots<i_p}{{\sum}} \lambda_I e_{i_1} \wedge \cdots \wedge e_{i_p}$. Then
$$\|x\|_{\wedge1,p} \geq |\langle l_{i_1} \wedge \cdots \wedge \l_{i_p},x\rangle |=|\lambda_I|$$
which gives
$$\|x\|_{\wedge2,p} \leq \underset{i_1<\dots<i_p}{{\sum}} |\lambda_I| \|e_{i_1} \| \cdots \| e_{i_p}\| \leq \underset{i_1<\dots<i_p}{{\sum}} \|x\|_{\wedge1,p} \leq \begin{pmatrix} kp \\ p \end{pmatrix} \|x\|_{\wedge1,p}.$$
\end{proof}

\begin{prop}
\label{cv exterieur}
The subset $\{ x_1\wedge\dots\wedge x_p ~ |~ x_1,\dots,x_p \in E \}$ of $\bigwedge\nolimits^pE$ is a complete metric space for $\|.\|_{\wedge,p}$.
\end{prop}

\begin{proof}
Let $\alpha_n=x_1^n \wedge\dots\wedge x_p^n$ be a Cauchy sequence for $\|.\|_{\wedge,p}$. If there are infinitely many $k$ such that $\alpha_k=0$, then $\alpha_n \underset{n \rightarrow +\infty}{\longrightarrow} 0$.
Else, for $n$ big enough we have $\alpha_n \neq 0$ hence $V_n=\textnormal{Span}(x_1^n,\dots,x_p^n)$ is a $p$-dimensional subspace of $E$. As $d_{\wedge,p}$ and $d_H$ are equivalent, $V_n$ is Cauchy in the complete metric space $(\mathcal{G}_p(E),d_H)$ and therefore converges to $V \in \mathcal{G}_p(E)$. Let $(x_1,\dots,x_p)$ be a basis of $V$ such that $\alpha=x_1\wedge\dots\wedge x_p$ satisfies $\|\alpha\|_{\wedge,p}=1$.  As $d_{\wedge,p}$ and $d_H$ are equivalent, there exists a sequence $\lambda_n \in \mathbb{U}$ such that $(\alpha_n - \lambda_n \alpha) \underset{n \rightarrow +\infty}{\longrightarrow} 0$. Then 

$$|\lambda_n - \lambda_p| \leq \|\lambda_n \alpha  - \lambda_p \alpha\|_{\wedge,p}=\|\lambda_n \alpha  - \alpha_n\|_{\wedge,p}+\|\alpha_n  - \alpha_p\|_{\wedge,p}+\|\lambda_p \alpha  - \alpha_p\|_{\wedge,p}.$$
This shows that the sequence $(\lambda_n)$ is Cauchy and therefore converges to some $\lambda \in \mathbb{U}$. This gives us $\alpha_n \underset{n \rightarrow +\infty}{\longrightarrow} \lambda \alpha$.
\end{proof}

\section{Regularity of $p$-cones and spectral gap theorem}
\label{spectral gap}

In this section, we note $|.|$ the standard hermitian norm on $\mathbb{C}^p$. We first define a notion of bounded aperture for $p$-cones, then use it to compare norms with our cone "distance". We finally use these comparisons to get a spectral gap result. Here, the norm $\|.\|_{\wedge,p}$ on $\bigwedge\nolimits^pE$ can denote any of the previous norms on $\bigwedge\nolimits^pE$, as we never make computations on more than 2 decomposable tensors.

\begin{defi}
\label{aperture-def}
We say that a $p$-cone $\mathcal{C} \subset E$ is of aperture bounded by $K>0$ if there exists a continuous linear map $m:E \rightarrow \mathbb{C}^p$ such that $\forall u \in \mathcal{C}, \|m\|.\|u\| \leq K |m(u)|$. \\
We say that a $p$-cone $\mathcal{C} \subset E$ is of sectional aperture bounded by $K>0$ if for every $(V,W)$ $p$-dimensional spaces contained in $\mathcal{C}$, there exists a linear map $m_{V,W}:\textnormal{Span}(V,W) \rightarrow \mathbb{C}^p$ such that $\forall x \in \mathcal{C} \cap \textnormal{Span}(V,W), \|m_{V,W}\|.\|u\| \leq K |m_{V,W}(u)|$.
\end{defi}

\begin{lemma}
\label{aperture-1}
Let $\mathcal{C} \subset E$ be a $p$-cone of aperture bounded by $K>0$. Let $V,W$ be $p$-dimensional spaces contained in $\mathcal{C}$ and $m=m_{V,W}$ a linear map satisfying the inequality from Definition \ref{aperture-def}.
Then for all $x \in V^*$, there exists $y\in W^*$ such that 
$$\left\| \dfrac {x} {|m(x)|}-\dfrac {y} {|m(y)| }\right\| \leq \dfrac {K} {\|m\|}d_{\mathcal{C}}(V,W)$$
and even

$$\left\|x-y\right\| \leq \dfrac {K} {\|m\|}|m(x)|d_{\mathcal{C}}(V,W)$$
\end{lemma}

\begin{proof}
We renormalize $m$ so that $\|m\|=K$. For all $u \in V^*$, $|m(u)|\geq \frac{1}{K} \|m\|.\|u\|>0$, therefore $m_{|V}:V\rightarrow \mathbb{C}^p$ is injective, and therefore bijective. Similarly, $m_{|W}:W\rightarrow \mathbb{C}^p$ is bijective.
Let $x \in V^*$. Then there exists $y \in W^*$ such that $m(x)=m(y)\neq 0$. In particular $m(x)$ and $m(y)$ are colinear, which will allow us to follow the proofs of \cite{Rug10} and \cite{Dub09}. Let $x'=\frac{x}{|m(x)|}$ and $y'=\frac{y}{|m(y)|}$.

\emph{Case 1: with the gauge from \cite{Rug10}}\\
Let $u_\lambda = (1+\lambda)x'+(1-\lambda)y'$ pour $\lambda \in \mathbb{C}$. Then $|m(u_\lambda)|=|(1+\lambda)\frac{m(x)}{|m(x)|}+(1-\lambda)\frac{m(x)}{|m(x)|}|=2$. When $u_\lambda \in \mathcal{C}$, we get (using Definition \ref{aperture-def}):

$$\lambda\|x'-y'\| \leq \|u_\lambda\| + (\|x'\|+\|y'\|) \leq 2 + 1 + 1 \leq 4$$
Following the exact proof of Lemma 3.4 in \cite{Rug10}, we get the desired results.

\emph{Case 2: with the gauge from \cite{Dub09}}\\
This is the exact same proof and calculations as in \cite{Dub09}.
\end{proof}

We now extend this result on representatives of $p$-dimensional spaces in $\bigwedge\nolimits^p E$.

\begin{defi}
If $V$ is a complex Banach space and $m:V \rightarrow \mathbb{C}^p$ is a continuous linear map, we consider the map $\hat{m}: \bigwedge\nolimits^pV \rightarrow \mathbb{C}^p$ obtained from the alternating $p$-linear map $(x_1,\dots,x_p) \mapsto \det_{\mathbb{C}^P}(m(x_1),\dots,m(x_p))$.
\end{defi}

\begin{lemma}
\label{aperture-2}
Let $\mathcal{C} \subset E$ be a $p$-cone of aperture bounded by $K>0$. Let $V,W$ be $p$-dimensional spaces contained in $\mathcal{C}$ and $m=m_{V,W}$ a linear map satisfying the inequality from Definition \ref{aperture-def}.
If $(x_1,\dots,x_p)$ is a basis of $V$ and  $(y_1,\dots,y_p)$ is a basis of $W$, then
$$\left\| \dfrac {x_{1}\wedge \ldots \wedge x_{p}} {\hat {m}(x_{1}\wedge \ldots \wedge x_{p}) } - \dfrac {y_{1}\wedge \ldots \wedge y_{p}} {\hat {m}(y_{1}\wedge \ldots \wedge y_{p}) } \right\| \leq p.p! \dfrac{K^p}{\|m\|^p}d_{\mathcal{C}}(V,W)$$
\end{lemma}

\begin{proof}
Given two basis $(x_1,\dots,x_p)$ and $(x_1',\dots,x_p')$ of V, $x_{1}\wedge \ldots \wedge x_{p}$ and $x_{1}' \wedge \ldots \wedge x_{p}'$ are colinear (as it is for two basis of $W$). Hence we only have to show the result for one particular basis of $V$ and one particular basis of $W$.

Let $(e_1,\dots,e_p)$ be the canonical basis of $\mathbb{C}^p$. As in the proof of Lemma \ref{aperture-1}, $m_{|V}:V \rightarrow \mathbb{C}^p$ is bijective. Then there exists a basis $(x_1,\dots,x_p)$ of $V$ such that $(m(x_1),\dots,m(x_p))=(e_1,\dots,e_p)$. As in Lemma \ref{aperture-1}, we choose $(y_1,\dots,y_p)$ in $W$ such that $m(x_i)=m(y_i)=e_i$ for $1\leq i \leq p$. In particular $(y_1,\dots,y_p)$ is a basis of $W$. Note that $|m(x_i)|=|m(y_i)|=1$ for all $i$, and $\hat{m}(x_1,\dots,x_p)=\hat{m}(y_1,\dots,y_p)=1$. Using Lemma \ref{aperture-1}, we get: for all $i$, $\left\| x_i- y_i\right\| \leq \dfrac {K} {\|m\|}d_{\mathcal{C}}(V,W)$. Using this inequality and the fact that $\|x_i\| \leq \frac{K}{\|m\|}$, we get:

$$
\begin{array}{rcl}
\left\| x_{1}\wedge \ldots \wedge x_{p} - y_{1}\wedge \ldots \wedge y_{p} \right\| &\leq & \|(x_1 - y_1) \wedge x_2 \wedge \dots \wedge x_p\| +  \|y_1 \wedge (x_2 \wedge \dots \wedge x_p - y_2 \wedge \dots \wedge y_p)\| \\
&\leq &p\|(x_1 - y_1) \| \| x_2 \wedge \dots \wedge x_p\| + p \|y_1 \| \|x_2 \wedge \dots \wedge x_p - y_2 \wedge \dots \wedge y_p\| \\
&\leq & p \dfrac{K}{\|m\|} d_{\mathcal{C}}(V,W) . (p-1)! \dfrac{K^{p-1}}{\|m\|^{p-1}} +p \dfrac{K}{\|m\|} \|x_2 \wedge \dots \wedge x_p - y_2 \wedge \dots \wedge y_p\| \\
&\leq &p! \dfrac{K^p}{\|m\|^p} d_{\mathcal{C}}(V,W) + p\dfrac{K}{\|m\|} \|x_2 \wedge \dots \wedge x_p - y_2 \wedge \dots \wedge y_p\|

\end{array}
$$

The desired result follows from a simple recursion.
\end{proof}

\begin{ex}
Let $F$ be a $p$-dimensional subspace of $E$, and $G$ a closed supplement of $F$ in $E$. Let $\pi:E \rightarrow F$ be the projection on $F$ parallel to $G$ and $a>0$. Then the cone $\mathcal{C}_{\pi,a}$ is of bounded aperture. Indeed, let $\phi:F \rightarrow \mathbb{C}^p$ be an isomorphism and define the continuous linear map $m:E \rightarrow \mathbb{C}^p$ as $m=\phi \circ \pi$. For $x=x_F+x_G \in \mathcal{C}_{\pi,a}$, $|m(x)|=|\phi(x_F)|\geq \|\phi^{-1}\|^{-1} \|x_F\| \geq \frac{\|\phi^{-1}\|^{-1}}{1+a} \|x\|$. 
\end{ex}

\begin{rem}
In the previous example, the isomorphism $\phi$ can be choosen such that $\|\phi\|$ and $\|\phi^{-1}\|$ only depends on the dimension $p$.
\end{rem}
\begin{prop}
\label{fixed-space}
Let $\mathcal{C} \subset E$ be a $p$-cone of aperture bounded by $K>0$. Le $T \in \mathcal{L}(E)$ such that $T\mathcal{C}^* \subset \mathcal{C}^*$ and diam${}_{\mathcal{C}}(T\mathcal{C})=\Delta< \infty$. Let $\eta<1$ be the contraction coefficient given by Proposition \ref{contraction}. Then:

\begin{enumerate}[(i)]
\item There exists a unique $p$-dimensional subspace $V$ contained in $\mathcal{C}$ fixed by $T$. Given a representative $h \in \bigwedge\nolimits^pE$ of $V$ with norm $1$, $h$ is an eigenvector for $\hat{T}$ for some eigenvalue $\lambda \in \mathbb{C}^*$ : $\hat{T}h=\lambda h$.
\item There exist constants $R,C < \infty$ and a map $b: \{ x_1 \wedge \dots \wedge x_p ~ |~ \textnormal{Span}(x_1,\dots,x_p) \subset \mathcal{C} \} \rightarrow \mathbb{C}$ such that for all $(x_1,\dots,x_p)$ with $\textnormal{Span}(x_1,\dots,x_p) \subset \mathcal{C}$ and all $n \in \mathbb{N}^*$,
$$ \|\lambda^{-n}\hat{T}(x_1 \wedge \dots \wedge x_p) - c(x_1 \wedge \dots \wedge x_p).h\| \leq C \eta^n \|x_1 \wedge \dots \wedge x_p\|,$$
$$|c(x_1 \wedge \dots \wedge x_p)| \leq R \|x_1 \wedge \dots \wedge x_p\|.$$

\end{enumerate}
%meilleure notation possible ? Peut-être noter \hat(x) pour x_1 \wedge ... \wedge x_p ?
\end{prop}

\begin{proof}
\begin{enumerate}[(i)]
\item
Let $V_0=\textnormal{Span}(x_1^0,\dots,x_p^0)$ be a $p$-dimensional subspace contained in $\mathcal{C}$ and $\hat{x_0}=x_1^0 \wedge \dots \wedge x_p^0/\|x_1^0 \wedge \dots \wedge x_p^0\|$. Let $V_n=T^nV_0$. We construct recursively a sequence $(\hat{x}_n)$: given $\hat{x_n}$, Definition \ref{aperture-def} gives us a functional $m_n$ of norm $K$ adapted with $V_n$ and $V_{n+1}$. Let $\lambda_n = \hat{m_n}(\hat{T}\hat{x_n})/\hat{m_n}(\hat{x_n}) \in ]0, \|T\|^pK^p]$. We define:
%encadrement de \|\hat{T}\| à faire
$$\hat{x_{n+1}} = \dfrac{\lambda_n^{-1}\hat{T}\hat{x_n}}{\|\lambda_n^{-1}\hat{T}\hat{x_n}\|} \in \textnormal{Span}(\hat{T}^n\hat{x_0})$$
It is a representative of $V_{n+1}$ in $\bigwedge\nolimits^pE$. By Lemma \ref{aperture-2}, we have:
$$\left\| \dfrac{\hat{x_n}}{\hat{m_n}(\hat{x_n})} - \dfrac{\hat{T}\hat{x_n}}{\hat{m_n}(\hat{T}\hat{x_n})} \right\| \leq p.p! d_{\mathcal{C}}(V_n,V_{n+1}) \leq p.p! \Delta \eta^{n-1}$$

As $|\hat{m_n}(\hat{x_n})| \leq K^p$, we get $\|\hat{x_n} - \lambda_n^{-1}\hat{T}\hat{x_n}\| \leq p.p! \Delta K^p \eta^{n-1}$. Combined with the fact that $\|\hat{x_n}\|=1$, we get $\|\hat{x_n} - \hat{x_{n+1}}\| \leq 2p.p! \Delta K^p \eta^{n-1}$. Thus $(\hat{x_n})$ is Cauchy and so is $(V_{n})$, and by Proposition \ref{cv exterieur}, $\hat{x_n}\underset{n \rightarrow \infty}{\rightarrow} h$ and $ V_n \underset{n \rightarrow \infty}{\rightarrow} V$ contained in $\mathcal{C}$ as $\mathcal{C}$ is closed.

There are two different methods to prove that $V$ is a fixed space:

\emph{Method 1.} $|\lambda_n - \lambda_{n+1}| \leq (\hat{T}-\lambda_n)\hat{x_n}+ (\lambda_{n+1}-\hat{T})\hat{x_{n+1}}+(\hat{T}-\lambda_n)(\hat{x_{n+1}}-\hat{x_n}).$
Given our inequalities and the fact that $\lambda_n \leq \|T\|^pK^p$, we get 
$$|\lambda_n - \lambda_{n+1}| \leq (K^p+\eta K^p+(2+2K^p)).p.p!\Delta \|T\|pK^{p} \eta^{n-1}$$
so that $\lambda_n \underset{n \rightarrow \infty}{\rightarrow} \lambda \in \mathbb{C}$. As $|\hat{x_n} - \lambda_n^{-1}\hat{T}\hat{x_n}\| \underset{n \rightarrow \infty}{\rightarrow} 0$, we get $\|\lambda h - \hat{T}h\|=0$, hence $\hat{T}h=\lambda h$. As $T\mathcal{C}^* \subset \mathcal{C}^*$, this proves $\lambda \neq 0$ and $V$ is fixed by $T$.

\emph{Method 2.} Let $x \in V$. There exist $x_n \in V_n$ such that $x_n  \underset{n \rightarrow \infty}{\rightarrow} x$, and this sequence is bounded by a $M>0$. Then $Tx_n  \underset{n \rightarrow \infty}{\rightarrow} Tx$ and therefore
$$d(Tx,V) \leq d(Tx,Tx_n)+d(Tx_n,V) \leq d(Tx,Tx_n)+\|Tx_n\|d_H(V_n,V) \leq \underbrace{d(Tx,Tx_n)}_{\underset{n \rightarrow \infty}{\rightarrow} 0} +\|T\|M \underbrace{d_H(V_n,V)}_{\underset{n \rightarrow \infty}{\rightarrow} 0}$$
Thus $d(Tx,V)=0$ and $Tx \in V$, proving that $V$ is fixed by $T$.

The space $V$ is unique: if $W$ is such that $TW = W$, then $d_{\mathcal{C}}(V,W)=d_{\mathcal{C}}(TV,TW) \leq \eta d_{\mathcal{C}}(V,W)$, hence $d_{\mathcal{C}}(V,W)=0$ and $V=W$.

\item The proof of the second part of our proposition follows the method from \cite{Rug10}. It is the same computations, obtained by replacing the '$x_n$' in \cite{Rug10} with $\hat{x_n}=T^nx^1 \wedge \dots \wedge T^nx^p$.
\end{enumerate}
\end{proof}

We need a final regularity assumption on our cone $\mathcal{C}$ to state the spectral gap theorem. Our cone must be $p$-reproducing, which means it must give us information about all $p$-dimensional directions.

\begin{defi}
\label{reproducing}
We say that a $p$-cone $\mathcal{C} \subset E$ is $p$-reproducing if there exists $k \in \mathbb{N}^*$ and $A>0$ such that for all $(x_1,\dots,x_p) \in E^p$, there exists $((x_1^i,\dots,x_p^i))_{1\leq i \leq k}$ such that
\begin{align*}
i \in \{1,\dots,k\}, \textnormal{Span}(x_1^i,\dots,x_p^i) \subset \mathcal{C} \\
x_1 \wedge \dots \wedge x_p = \sum\limits_{i=1}\limits^{k} x_1^i \wedge \dots \wedge x_p^i \\
\sum\limits_{i=1}\limits^{k} \|x_1^i \wedge \dots \wedge x_p^i \| \leq A \| x_1 \wedge \dots\wedge x_p\| 
\end{align*} 
\end{defi}

\begin{ex}
The cone $\mathcal{C}_{\pi,a}$ is reproducing. Indeed, let $(x_1,\dots,x_p) \in E^p$ be a right decomposition and let $x_k=x_k^F+x_k^G$ be the decomposition of $x^i$ in $E=F \oplus G$. Then 
$$x_1 \wedge \dots \wedge x_p = \sum\limits_{(a_1,\dots,a_p) \in \{F,G\}^p} x_1^{a_1} \wedge \dots \wedge x_p^{a_p}.$$
Let's estimate the terms $x_1^{i_1} \wedge \dots \wedge x_p^{i_p}$. Up to a change of sign, such a term can be written as \\ $x_{i_1}^{F} \wedge \dots \wedge x_{i_k}^F\wedge x_{j_1}^{G} \wedge \dots \wedge x_{j_{p-k}}^G$. Let's prove by induction  the following: \\ for all $n \in \{0,\dots,p\}$, $x_{i_1}^{F} \wedge \dots \wedge x_{i_k}^F\wedge x_{j_1}^{G} \wedge \dots \wedge x_{j_{n}}^G$ can be decomposed into a sum of $2^n$ decomposable tensors $\alpha_1, \dots, \alpha_{2^n}$ such that the subspaces represented by the $\alpha_i$ are in $\mathcal{C}_{\pi,\frac{a}{2^{p-n}}}$ and $\|\alpha_i\| \leq C_n \|x_{i_1}\| \dots \| x_{i_k}\| \|x_{j_1} \| \dots \|x_{j_{n}}\|$ for some constant $C_n$. \\

The result is obviously true for $n=0$. Let's assume the result is true for $n  \in \{0,\dots,p\}$. Then $x_{i_1}^{F} \wedge \dots \wedge x_{i_k}^F\wedge x_{j_1}^{G} \wedge \dots \wedge x_{j_{n+1}}^G=\underset{i=0}{\overset{2^n}{\sum}}\alpha_i \wedge  x_{j_{n+1}}^G$. We are going to decompose $\alpha_i \wedge  x_{j_{n+1}}^G$ into 2 decomposable tensors. Let $W_i \subset \mathcal{C}_{\pi,\frac{a}{2^{p-n}}}$ be the $i$-dimensional space represented by $\alpha_i$. There exists a $p$-dimensional space $W \subset \mathcal{C}_{\pi,\frac{a}{2^{p-n}}}$ such that $W_i \subset W$. Let $y \in W$ with $\|y\|=1$ such that $d_H(W_i,\textnormal{Span}(y)) \geq 1/2$. Then 
$$\alpha_i \wedge  x_{j_{n+1}}^G=\alpha_i \wedge  (x_{j_{n+1}}^G-\frac{2^{p-n}}{a}\|x_{j_{n+1}}^G\|y)+\alpha_i \wedge  \frac{2^{p-n}}{a}\|x_{j_{n+1}}^G\|y.$$
The $(p+1)$-dimensional space represented by the second term is a subset of $W \subset \mathcal{C}_{\pi,\frac{a}{2^{p-n}}}$ and 
$$\begin{array}{rcl}
\left\|\alpha_i \wedge  \frac{2^{p-n}}{a}\|x_{j_{n+1}}\|y\right\| &\leq  &\|\alpha_i\| .\frac{2^{p-n}}{a}\|x_{j_{n+1}}^G\| \leq  C_n \|x_{i_1}\| \dots \| x_{i_k}\| \|x_{j_1} \| \dots \|x_{j_{n}}\|. \frac{2^{p-n}}{a}\|x_{j_{n+1}}^G\| \\
 &\leq &\frac{2^{p-n}}{a}C_n \|\pi_{F//G}\|\|x_{i_1}\| \dots \| x_{i_k}\| \|x_{j_1} \| \dots \|x_{j_{n}}\| \|x_{j_{n+1}}\|. \end{array}$$
Let's now show that the first term $\alpha_i \wedge  (x_{j_{n+1}}^G-2a\|x_{j_{n+1}}^G\|y)$ represents a subspace contained in $\mathcal{C}_{\pi,\frac{a}{2^{p-n-1}}}$. Let's $u=u_i+ \lambda  (x_{j_{n+1}}^G-\frac{2^{p-n}}{a}\|x_{j_{n+1}}^G\|y)$ in this subspace, with $u_i \in W_i$. As  $W_i \subset \mathcal{C}_{\pi,\frac{a}{2^{p-n}}}$, $u_i=u_i^F+u_i^G$ with $\frac{a}{2^{p-n}} \|u_i^F\| \geq \|u_i^G\|$. Then 
$$u=u_i+ \lambda  (x_{j_{n+1}}^G-\frac{2^{p-n}}{a}\|x_{j_{n+1}}^G\|y) = (u_i^F-\frac{2^{p-n}}{a}\lambda \|x_{j_{n+1}}^G\|y)+ (u_i^G+\lambda x_{j_{n+1}}^G).$$
Let's estimate the norms of the last two terms. As $d_H(W_i,\textnormal{Span}(y)) \geq 1/2$
$$\left\|u_i^F-\frac{2^{p-n}}{a}\lambda \|x_{j_{n+1}}^G\|y \right\| \geq \|u_i^F\| \textnormal{ and } \left \|u_i^F-\frac{2^{p-n}}{a}\lambda \|x_{j_{n+1}}^G\|y \right\| \geq \frac{2^{p-n}}{a} |\lambda| \|x_{j_{n+1}}^G\|,$$
Therefore$$\left\|u_i^F-\frac{2^{p-n}}{a}\lambda \|x_{j_{n+1}}^G\|y \right\| \geq \frac{1}{2}\|u_i^F\|+ \frac{1}{2}\frac{2^{p-n}}{a} |\lambda| \|x_{j_{n+1}}^G\|.$$
On the other hand, 
$$\left\|u_i^G+\lambda x_{j_{n+1}}^G\right\| \leq \left\|u_i^G\right\|  +|\lambda| \| x_{j_{n+1}}^G\| \leq \frac{2^{p-n}}{a}\|u_i^F\| + |\lambda| \| x_{j_{n+1}}^G\|\leq  \frac{2^{p-n-1}}{a} \left\|u_i^F-\frac{2^{p-n}}{a}\lambda \|x_{j_{n+1}}^G\|y \right\|. $$
That shows that $u \in \mathcal{C}_{\pi,a/2^{p-n-1}}$ and thus $\alpha_i \wedge  (x_{j_{n+1}}^G-2a\|x_{j_{n+1}}^G\|y)$ represents a subspace contained in $\mathcal{C}_{\pi,a/2^{p-n-1}}$. A computation similar to the first case show that 
$$\left\|\alpha_i \wedge  \|x_{j_{n+1}}\|y\right\| \leq(1+\frac{2^{p-n}}{a})C_n \|\pi_{F//G}\|\|x_{i_1}\| \dots \| x_{i_k}\| \|x_{j_1} \| \dots \|x_{j_{n}}\| \|x_{j_{n+1}}\|.$$

To conclude, our induction for $n=p-k$ maximal shows that we can decompose $$x_{i_1}^{F} \wedge \dots \wedge x_{i_k}^F\wedge x_{j_1}^{G} \wedge \dots \wedge x_{j_{p-k}}^G=\underset{i=0}{\overset{2^{p-k}}{\sum}}\alpha_i$$ where the $\alpha_i$ represent subspaces of $\mathcal{C}_{\pi,a}$ and $\|\alpha_i\| \leq C  \|x_1\| \dots \| x_p\| $ for some constant $C$ (depending on $p$, $\pi$ and $a$). As we had 
$$x_1 \wedge \dots \wedge x_p = \sum\limits_{(a_1,\dots,a_p) \in \{F,G\}^p} x_1^{a_1} \wedge \dots \wedge x_p^{a_p},$$
$x_1 \wedge \dots \wedge x_p$ can be decomposed in a giant sum of such $\alpha_i$. As we can choose $(x_1 , \dots ,x_p)$ to be a right decomposition without changing $x_1 \wedge \dots \wedge x_p$, we get $x_1 \wedge \dots \wedge x_p=\underset{i \in I}{\sum}\alpha_i$ (with $|I| \leq 4^p$) where $\alpha_i$ represent subspaces of $\mathcal{C}_{\pi,a}$ and $$\underset{i \in I}{\sum}\|\alpha_i\| \leq |I| C  \|x_1\| \dots \| x_p\| \leq 4^p C p!  \|x_1\wedge \dots \wedge x_p\|.$$

This proves that the cone $\mathcal{C}_{\pi,a}$ is reproducing. 
\end{ex}
%%% ATTENTION NORME WEDGE2 DANS L'EXEMPLE
%%% ON N'EXPLIQUE PAS D'OU SORT LE (2+a)
\begin{theorem}
Let $\mathcal{C} \subset E$ be a reproducing $p$-cone of aperture bounded by $K>0$. Le $T \in \mathcal{L}(E)$ such that $T\mathcal{C}^* \subset \mathcal{C}^*$ and diam${}_{\mathcal{C}}(T\mathcal{C})=\Delta< \infty$. Let $\eta<1$ be the contraction coefficient given by Proposition \ref{contraction}. Then $T$ has a spectral gap, i.e. there exists a $p$-dimensional subspace $V$ and a closed supplement $W$ of $V$, both stable by $T$, such that $\sup Sp(T_{|W}) \leq \eta. \inf Sp(T_{|V})$.
\end{theorem}

\begin{proof}
Let $(x_1,\dots,x_p) \in E^p$ and $\hat{x}=x_1 \wedge \dots \wedge x_p$. Then Definition \ref{reproducing} gives us a decomposition $\hat{x}=\sum\limits_{i=1}\limits^{k} \hat{x^i}$ where $\hat{x^i}=x_1^i \wedge \dots \wedge x_p^i$. Using Proposition \ref{fixed-space}, we get 
$$ \|\lambda^{-n}\hat{T}(\hat{x^i}) - c(\hat{x^i}).h\| \leq C \eta^n \|\hat{x^i}\|$$
hence 
$$ \|\lambda^{-n}\hat{T}(\hat{x}) - (c(\hat{x^1})+\dots c(\hat{x^k})).h\| \leq C \eta^n \sum\limits_{i=1}\limits^{k} \|\hat{x^i}\| \leq C.A \eta^n \|\hat{x}\|$$

We set $c(\hat{x})=c(\hat{x^1})+\dots c(\hat{x^k})$. Then $|c(\hat{x})| \leq R.A.\|\hat{x}\|$, and letting $n \rightarrow \infty$ in the previous inequality shows that $c(\hat{x})$ only depends on $\hat{x}$ and not on the decomposition. Given the linearity of $\hat{T}$, we can extend $c$ to a \emph{linear} map $c: \bigwedge\nolimits^pE \rightarrow \mathbb{C}$ such that for all $\hat{x}=x_1 \wedge \dots \wedge x_p$,
$$\|\lambda^{-n}\hat{T}(\hat{x}) - c(\hat{x}).h\| \leq CA \eta^n \|\hat{x}\|. ~  ~ ~ ~ ~ ~ (*)$$

Let $V \subset \mathcal{C}$ be the $p$-dimensional space fixed  by $T$. Let $(h_1,\dots,h_p)$ be a basis of $V$ in which $T_{|V}$ is triangular, and $(\lambda_1,\dots,\lambda_p)$ the diagonal coefficients (which are the eigenvalues of $T$ with multiplicity). We may assume $|\lambda_1| \geq \dots \geq |\lambda_p|$. We may multiply the map $c$ by a constant and assume $h=h_1 \wedge \dots \wedge h_p$. As $\hat{T} h= Th_1 \wedge \dots \wedge Th_p = \lambda_1 \dots \lambda_p. h_1 \wedge \dots \wedge h_p=\lambda_1 \dots \lambda_p.h$ and $Th=\lambda h$ by Proposition \ref{fixed-space}, we get $\lambda= \lambda_1 \dots \lambda_p$. Taking $\hat{x}=h$ and letting $n$ go to infinity in $(*)$ shows that $c(h)=1$.

Let's consider the linear maps $\pi_i: E \mapsto V$ defined by $\pi_i(x)=c(h_1 \wedge \cdots h_{i-1}\wedge x \wedge h_{i+1} \wedge h_p).h_i$. The map $\pi_i$ is a projection onto Span$(h_i)$, and furthermore $\pi_i \circ \pi_j = 0$ if $i \neq j$. Thus $\pi=\pi_1+\dots+\pi_p$ is a projection onto $V$ parallel to 
$\ker (p)= \{ x \in E ~ |~ \forall i, \pi_i(x)=0 \} = \{ x \in E ~ |~ \forall i, c(h_1 \wedge \cdots h_{i-1}\wedge x \wedge h_{i+1} \wedge h_p)=0 \}$.

From $(*)$, we get that $\forall (x_1, \dots,x_p) \in E^p, c(Tx_1 \wedge\dots \wedge Tx_p)=\lambda_1\dots\lambda_p c(x_1 \wedge \dots \wedge x_p)$. This gives for all $x \in E$:
$$\begin{array}{rcl}
c(h_1 \wedge \cdots h_{i-1}\wedge Tx \wedge h_{i+1} \wedge h_p)&=&\lambda_1^{-1}\dots\lambda_{i-1}^{-1}\lambda_{i+1}^{-1}\dots\lambda_p^{-1}c(Th_1 \wedge \cdots Th_{i-1}\wedge Tx \wedge Th_{i+1} \wedge Th_p)\\
&=& \lambda_ic(h_1 \wedge \cdots h_{i-1}\wedge x \wedge h_{i+1} \wedge h_p).
\end{array}$$
In particular, this shows that $\ker (p)$ is stable by $T$.

Let $x \in E$. We have:
$$
\begin{array}{rcl}
\|h_1 \wedge \dots \wedge h_{p-1} \wedge (T^n(Id-p)x) \| &= &\|h_1 \wedge \dots \wedge h_{p-1} \wedge T^nx -c(h_1 \wedge \dots\wedge h_{p-1} \wedge x)\lambda_p^{n}n h_1 \wedge \dots \wedge h_p\| \\
&=&\|\lambda_1^{-n}\dots\lambda_{p-1}^{-n}T^nh_1 \wedge \dots \wedge T^nh_{p-1} \wedge T^nx \\
& &  -\lambda_p^{n} c(h_1 \wedge \dots \wedge h_{p-1} \wedge x)h_1 \wedge \dots \wedge h_p\| \\
&=&|\lambda_p|^n.\|\lambda^{-n}\widehat{T}^n(h_1 \wedge \dots \wedge h_{p-1}\wedge x)- c(h_1 \wedge \dots \wedge h_{p-1} \wedge x)h_1 \wedge \dots \wedge h_p\| \\
&\leq & CA |\lambda_p|^n \eta^n \|h_1 \wedge \dots \wedge h_{p-1} \wedge x\| ~ ~ ~ ~ ~ (**)
\end{array}
$$

To conclude our estimation we need the following lemma:

\begin{lemma}
Let $V$ be a $p$-dimensional subset of $E$ and $W$ a closed supplement of $E$. Let $h_1,\dots,h_{p-1}$ be independant vectors of $V$. Then there exists $B >0$ such that:
$$\forall y \in W, \|h_1 \wedge \dots \wedge h_{p-1} \wedge y \| \geq B \|y\|.$$
\end{lemma}

\begin{proof}
Let $l_1,\dots,l_{p-1}$ be linear forms such that $l_i(h_i)=1$, $l_i(h_j)=0$ if $i\neq j$, and $l_{i|W}=0$. Let $y \in W$ and let $l$ be a linear form on $V+\textnormal{Span}(y)$ such that $l_{|V}=0$ and $l(y)=\|y\|$. Let's estimate $\|l\|$. For $x \in V, \mu \in \mathbb{C}$, we have $l(x+\mu y)=\|\mu y\| = \|\pi_{W \parallel V} (x+\mu y) \| \leq \|\pi_{W \parallel V} \| \|x+\mu y \| $. Thus $\|l\|\leq \|\pi_{W \parallel V} \|$.

Then $$
\begin{array}{rcl}

|h_1 \wedge \dots \wedge h_{p-1} \wedge y \| &\geq &\langle \frac{l_1}{\|l_1\|} \wedge \frac{l_{p-1}}{\|l_{p-1}\|} \wedge \frac{l}{\|l\|}, h_1 \wedge \dots \wedge h_{p-1} \wedge y \rangle \\
&\geq &\dfrac{1}{\|l_1\| \dots \|l_{p-1}\| \|l\|} \left| \det \begin{pmatrix}
1 & 0 & \cdots & 0\\
0 &\ddots & \ddots & \vdots \\
\vdots & \ddots & 1 &0\\
0 & \cdots & 0 & \|y\|
\end{pmatrix} \right| \\
&\geq &\dfrac{1}{\|l_1\| \dots \|l_{p-1}\| \|\pi_{W \parallel V} \|} \|y\|.
\end{array}$$
The constant $B=\dfrac{1}{\|l_1\| \dots \|l_{p-1}\| \|\pi_{W \parallel V} \|} $ (independent of $y \in W$) works. 
\end{proof}

We now apply this lemma to $(**)$ with $W = \ker(p)$ and $y = (T^n(Id-p)x)$:
$$\|T^n(Id-p)x\| \leq \frac{1}{B}\|h_1 \wedge \dots \wedge h_{p-1} \wedge (T^n(Id-p)x) \| \leq \frac{CA}{B} (\eta|\lambda_p|)^n \|h_1 \wedge \dots \wedge h_{p-1} \wedge x\|.$$

This proves $\sup Sp(T_{|W})\leq \eta |\lambda_p| = \eta \inf Sp(T_{|V})$.

\end{proof}

\begin{prop}
The application $c : \bigwedge\nolimits^pE \rightarrow \mathbb{C}$ which satisfies $$\forall \hat{x} \in \bigwedge\nolimits^pE , \|\lambda^{-n}\hat{T}(\hat{x}) - c(\hat{x}).h\| \leq CA \eta^n \|\hat{x}\|$$
can be written as a simple tensor $l_1 \wedge \dots \wedge l_p$, where $l_i: x \mapsto c(h_1 \wedge \cdots \wedge h_{i-1}\wedge x \wedge h_{i+1} \wedge \dots \wedge h_p)$.
\end{prop}

\begin{proof}
As seen in the previous proof, we have for all $x \in E$, $l_i(Tx)=\lambda^i x$. Hence for all $\hat{x} \in \bigwedge\nolimits^pE$, $l_1 \wedge \dots \wedge l_p (\hat{T}\hat{x})=\lambda_1 \dots \lambda_p. l_1 \wedge \dots \wedge l_p (\hat{x})= \lambda .l_1 \wedge \dots \wedge l_p (\hat{x})$, and therefore $l_1 \wedge \dots \wedge l_p (\lambda^{-n}\hat{T}^n\hat{x})=l_1 \wedge \dots \wedge l_p (\hat{x})$. As $\|\lambda^{-n}\hat{T}(\hat{x}) - c(\hat{x}).h\| \tendn 0$, we get 
$$l_1 \wedge \dots \wedge l_p (\hat{x}) = l_1 \wedge \dots \wedge l_p (\lambda^{-n}\hat{T}^n\hat{x}) \tendn l_1 \wedge \dots \wedge l_p (c(\hat{x}).h)=c(\hat{x}).$$
\end{proof}

%%%%%%%%%%%%%%%%%%%%%%%%%%%%%%%%

\section{Cone contraction for random products of operators}
\label{random-section}

In this section, $\|.\|_{\wedge,p}$ will denote the norm $\|.\|_{\wedge2,p}$ on $\bigwedge\nolimits^pE$. For $x \in E$ and $r \geq 0$, we will denote with $B(x,r)$ the ball of center $x$ and radius $r$.

\begin{defi}
Given a $p$-cone $\mathcal{C}$, we denote $G_p(\mathcal{C}):=\{ W \subset \mathcal{C} ~ |~ W \textnormal{ is a sub vector space of dimension } p \}$. Given $m$ as in Definition \ref{aperture-def}, we define $$\widehat{\mathcal{C}}_{m=1}:=\{x_1\wedge \cdots \wedge x_p ~ |~ \widehat{m}(x_1\wedge \cdots \wedge x_p)=1 \textnormal{ and } \Span(x_1, \cdots, x_p) \subset \mathcal{C} \}$$

where $\widehat{m}(x_1\wedge \cdots \wedge x_p)=\det(m(x_1),\dots,m(x_p))$. 
\end{defi}

In the next propositions, we will prove that we can identify $G_p(\mathcal{C})$ and $\widehat{\mathcal{C}}_{m=1}$. That means $G_p(\mathcal{C})$ considered as a submanifold of $G_p(E)$ and $\widehat{\mathcal{C}}_{m=1}$ considered a a submanifold of $\bigwedge\nolimits^p E$ share the same analytical structure. Then we'll consider the $p$-dimensional subspaces in $\mathcal{C}$ as elements of $\widehat{\mathcal{C}}_{m=1}$ for analytical purposes, and as elements of $G_p(\mathcal{C})$ for geometric purposes.

\begin{prop}
\label{homeo grassmannienne}
The application
$$G: \left\{\begin{array}{rcl}
\widehat{\mathcal{C}}_{m=1} &\longrightarrow &G_p(\mathcal{C}) \\
x_1 \wedge \cdots \wedge x_p &\mapsto &\Span(x_1, \cdots, x_p)
\end{array}\right.$$
is a homeomorphism.
\end{prop}

\begin{proof}
If $\Span(x_1, \cdots, x_p)=\Span(y_1, \cdots, y_p)$, then $x_1\wedge \cdots \wedge x_p= \lambda y_1\wedge \cdots \wedge y_p$ for some $\lambda \in \mathbb{C}$. Then, as $\widehat{m}(x_1\wedge \cdots \wedge x_p)=\widehat{m}(y_1\wedge \cdots \wedge y_p)=1$, we get $\lambda=1$ and $x_1\wedge \cdots \wedge x_p= y_1\wedge \cdots \wedge y_p$. Hence $G$ is injective.

Let $W \in G_p(\mathcal{C})$ and $x_1,\cdots,x_p$ such that $W = \Span(x_1,\cdots,x_p)$. Then $x_1\wedge \cdots \wedge x_p \neq 0$, and as $\Span(x_1,\cdots,x_p) \subset \mathcal{C}$ we get $\widehat{m}(x_1\wedge \cdots \wedge x_p) \neq 0$ (as in Lemma \ref{aperture-2}). Hence $W=G(\frac{x_1}{\widehat{m}(x_1\wedge x_2\wedge \cdots \wedge x_p)} \wedge \cdots \wedge x_p)$ and $G$ is surjective.

Let $x_1 \wedge \cdots \wedge x_p, y_1 \wedge \cdots \wedge y_p \in \widehat{\mathcal{C}}_{m=1}$.Then $\frac{1}{\widehat{m}} \leq \| x_1 \wedge \cdots \wedge x_p \|_{\wedge} \leq 1$ and we get :
$$
\begin{array}{ll}
d_{Gr}(G(x_1 \wedge \cdots \wedge x_p),G(y_1 \wedge \cdots \wedge y_p)) &\leq Cd_{\wedge}(G(x_1 \wedge \cdots \wedge x_p),G(y_1 \wedge \cdots \wedge y_p)) \\
&\leq 2C .\underset{\lambda \in \mathbb{C}}{\inf}\left\| \frac{x_1 \wedge \cdots \wedge x_p}{\|x_1 \wedge \cdots \wedge x_p\|} - \lambda \frac{y_1 \wedge \cdots \wedge y_p}{\|y_1 \wedge \cdots \wedge y_p\|}\right\| \\
&\leq 2C \|\widehat{m}\| .\underset{\lambda \in \mathbb{C}}{\inf}\| x_1 \wedge \cdots \wedge x_p - \lambda y_1 \wedge \cdots \wedge y_p \| \\
&\leq 2C \|\widehat{m}\|.\| x_1 \wedge \cdots \wedge x_p - y_1 \wedge \cdots \wedge y_p \|
\end{array}$$
Hence $G$ is Lispchitz.

We denote $x=x_1 \wedge \cdots \wedge x_p$, $y=y_1 \wedge \cdots \wedge y_p$, $x'=\frac{x_1 \wedge \cdots \wedge x_p}{\|x_1 \wedge \cdots \wedge x_p\|}$ and $y'=\frac{y_1 \wedge \cdots \wedge y_p}{\|y_1 \wedge \cdots \wedge y_p\|}$. Note that $x=\frac{x'}{\widehat{m}(x')}$ and $y=\frac{y'}{\widehat{m}(y')}$. By Theorem \ref{metriques-equivalentes}, there exist a constant $C>0$ and $\lambda \in \mathbb{C}$ with $|\lambda|=1$ such that $\|x' - \lambda y'\|_{\wedge} \leq C d_{Gr}(G(x),G(y))$. Hence

$$\begin{array}{rll}
\|x-y\|_{\wedge} &\leq \| \frac{x'-\lambda y'}{\widehat{m}(x')} \| + \|\frac{\lambda y'}{\widehat{m}(x')}-\frac{y'}{\widehat{m}(y')}\| \\
&\leq \frac{K^p}{\|m\|^p} \|x'-\lambda y'\| + \left| \frac{\lambda \widehat{m}(y')-\widehat{m}(x')}{\widehat{m}(x')\widehat{m}(y')} \right| .\|y'\| \\
&\leq  \frac{K^p}{\|m\|^p} \|x'-\lambda y'\| +\left| \frac{ \widehat{m}(\lambda y'-x')}{\widehat{m}(x')\widehat{m}(y')} \right| \\
&\leq \frac{K^p}{\|m\|^p} \|x'-\lambda y'\| +\frac{ K^{2p}}{\|m\|^p}\|x'-\lambda y'\| \\ 
&\leq \frac{K^p(K^p+1)}{\|m\|^p} C.d_{Gr}(G(x),G(y))

\end{array}$$
and $G^{-1}$ is Lipschitz. Therefore $G$ is a homeomorphism.
\end{proof}

We now show that the two structures also share the same analytical structure.

\begin{prop}
The map $G^{-1} : G_p(\mathcal{C}) \rightarrow \bigwedge\nolimits^pE$ is an embedding, where $G_p(\mathcal{C})$ is given its analytical structure from $G_p(E)$.
\end{prop}

\begin{proof}
Let $W_0 \in  G_p(\mathcal{C})$, and $(x_1, \cdots,x_p)$ an adapted decomposition of $W_0$. Given $F_0$ a  closed complement of $W_0$ in $E$, a local chart around $W_0$ is given by $L \in \mathcal{L}(W_0,F_0) \mapsto \graph~ L$ (in particular, $W_0=\graph~ L_0$ where $L_0 \equiv 0$). Let $H \in \mathcal{L}(W_0,F_0)$ and $W=\graph (L_0+H)$. Using the notations $x = x_1 \wedge \cdots \wedge x_p$ and $h_i=x_1 \wedge \cdots \wedge Hx_i \wedge \cdots\wedge x_p$, we get:

$$
\begin{array}{lcl}

G^1(W)-G^{-1}(W_0) &= & \dfrac{(x_1+Hx_1) \wedge \cdots \wedge (x_p+Hx_p)}{\widehat{m}((x_1+Hx_1) \wedge \cdots \wedge (x_p+Hx_p))} - \dfrac{x}{\widehat{m}(x)} \\
&= & \dfrac{x + \underset{i}{\sum} h_i + o(H)}{\widehat{m}(x)+\underset{i}{\sum} \widehat{m}(h_i) + o(H)} - \dfrac{x}{\widehat{m}(x)} \\
&= & (1-\underset{i}{\sum} \dfrac{\widehat{m}(h_i)}{\widehat{m}(x)} + o(H)) \dfrac{x}{\widehat{m}(x)} + \underset{i}{\sum} \dfrac{h_i}{\widehat{m}(x)} +o(H)- \dfrac{x}{\widehat{m}(x)}\\
&= &-\left(\underset{i}{\sum} \dfrac{\widehat{m}(h_i)}{\widehat{m}(x)^2}\right).x+\underset{i}{\sum} \dfrac{1}{\widehat{m}(x)}h_i + o(H)
\end{array}$$

Hence $G^{-1}$ is differentiable with derivative $H \mapsto -\left(\underset{i}{\sum} \dfrac{\widehat{m}(h_i)}{\widehat{m}(x)^2}\right)x+\underset{i}{\sum} \dfrac{1}{\widehat{m}(x)}h_i$. We now show that $G^{-1}$ is an immersion. Let $H \in \mathcal{L}(W_0,F_0)$ satisfying $-\left(\underset{i}{\sum} \dfrac{\widehat{m}(h_i)}{\widehat{m}(x)^2}\right)x+\underset{i}{\sum} \dfrac{1}{\widehat{m}(x)}h_i=0$. Denoting $I$ the set of indices $i$ such that $h_i \neq 0$, we verify easily that $\{h_i\}_{i \in I} \cup \{x\}$ is a linearly independent set of vectors of $\bigwedge\nolimits^pE$, thus $I = \emptyset$. Therefore for all $i$, $Hx_i \in W_0$  and $Hx_i \in F_0$, hence $Hx_i=0$ and $H=0$. The derivative of $G^{-1}$ is injective and  $G^{-1}$ is an homeomorphism onto its image by Proposition~\ref{homeo grassmannienne}, hence $G^{-1}$ is an embedding.

\end{proof}

In the rest of this section, we will consider a $p$-cone $\mathcal{C}$ of bounded aperture $K>0$ and with non-empty interior. For $\rho >0$ small enough, 
$$\mathcal{C}[\rho] := \{x \in \mathcal{C} ~ |~ B(x,\rho||x||) \subset \mathcal{C} \}$$
is not reduced to $\{0\}$: we now fix such a value of $\rho>0$. We also fix $\Delta<+\infty$ and write $\eta$ for the contraction constant given by Proposition \ref{contraction}.

\begin{defi}
\label{ensemble M}
We define $\mathscr{M} := \{ M \in \mathcal{L}(E) ~ |~ M(\mathcal{C}^*) \subset \mathcal{C}^*, \textnormal{diam}_{\mathcal{C}}(M(\mathcal{C}^*))\leq \Delta, M(\mathcal{C}) \subset \mathcal{C}[\rho] \}$
\end{defi}

We consider $(\Omega,\mu)$ a probability space and $\tau : \Omega \rightarrow \Omega$ a measurable $\mu$-ergodic application. Given a Banach space $F$, we denote by $\mathscr{B}(\Omega,F)$ the Banach space of ($\mu$-essentially) bounded mesurable maps equipped with the ($\mu$-essentially) uniform norm on $F$. If $A$ is a subset of $F$, we denote by $\mathscr{E}(\Omega,A)$ the set of Bochner-measurable maps from $\Omega$ into $A$.

For $(M_\omega)_{\omega \in \Omega} \in \mathscr{E}(\Omega,\mathscr{M})$, we denote by $M_\omega^{(n)}:=M_\omega \cdots M_{\tau^{n-1}\omega}$ the product of operators along the orbit of $\omega \in \Omega$. The goal of this section is to prove the following theorem:

\begin{theorem}
\label{theorem produit}
Let $t \in \mathbb{D} \mapsto M(t) \in \mathscr{E}(\Omega,\mathscr{M})$ be a map such that
\begin{enumerate}
\item $(t,\omega) \in \mathbb{D}\times \Omega \mapsto M_\omega(t) \in \mathcal{L}(E)$ is measurable and for all $\omega \in \Omega$, the map $t \mapsto M_\omega(t)$ is analytic.
\item $\sup \left\{ \dfrac{\|\bigwedge\nolimits^{p-1}M_\omega(t)\|}{\|\bigwedge\nolimits^{p}M_\omega(t)\|}\|\frac{d}{dt}M_\omega(t)\| : \omega \in \Omega, t \in \mathbb{D}\right\} < \infty$.
\item $\int_\Omega \left| \log \|M_\omega(0)\| \right| < \infty$.
\end{enumerate}
Then for all $t \in \mathbb{D}$, the limit
$$\chi(t)= \underset{n \rightarrow \infty}{\lim} \dfrac{1}{n} \log\|\widehat{M}_\omega^{(n)}\|$$
exists for $\mu$-a.e. $\omega \in \Omega$ and is ($\mu$-a.e.) independent of $\omega$. Moreover the map $t \mapsto \chi(t)$ is real-analytic and harmonic.

\end{theorem}

As $\mathcal{C}$ of bounded aperture $K>0$, we can find a linear map $m:E \rightarrow \mathbb{C}^p$ such that $\|m\|=1$ and $\forall x \in \mathcal{C}, \|x\| \leq K \|m(x)\|$. We fix such a $m$ for the rest of the section. Our first step is to understand and control $\|\widehat{M}\|_{\wedge,p}$. To do that, we use estimates on $\widehat{m}(\widehat{M}(v))$ where $v$ is representing $V \in \mathcal{C}[\rho]$.
\begin{defi}
\label{defipim}
For $M\in \mathcal{L}$ such that $M(\mathcal{C}^*) \subset \mathcal{C}^*$, we define $$\pi_M : \left\{\begin{array}{rcl}
\widehat{\mathcal{C}}_{m=1} &\longrightarrow &\widehat{\mathcal{C}}_{m=1} \\
x &\mapsto &\frac{\widehat{M}x}{\widehat{m}(\widehat{M}x)}
\end{array}\right.$$ which is differentiable (as a restriction of a differentiable application). Using the identification $\widehat{\mathcal{C}}_{m=1} \simeq G_p(\mathcal{C})$, $\pi_M$ can be seen as 
$$\pi_M : \left\{\begin{array}{rcl}
G_p(\mathcal{C})&\longrightarrow &G_p(\mathcal{C}) \\
W &\mapsto &M(W)\end{array}\right.$$
\end{defi}

\begin{lemma}
\label{contractionpim}
Soit $(M_n)_{n \in \mathbb{N}} \in \mathcal{M}^{\mathbb{N}}$. Let $M^{(n)}=M_n\circ\cdots\circ M_1$ and $\pi^{(n)}:=\pi_{M^{(n)}}=\pi_{M_n}\circ \cdots \circ \pi_{M_1}$. Let $v,w \in \widehat{\mathcal{C}}_{m=1}$. Then 
$$\| \pi^{(n)}(v) - \pi^{(n)}(w) \|_{\wedge,p} \leq C\eta^n$$
where $C>0$ only depends on $p$ and the definition of $\mathcal{M}$.
\end{lemma}

\begin{proof}
Let $V,W \in G_p(\mathcal{C})$ be the $p$-dimensional subspaces represented by $v$ and $w$. Using Lemma \ref{aperture-2}, we get 
$$\| \pi^{(n)}(v) - \pi^{(n)}(w) \|_{\wedge,p} \leq p.p! \dfrac{K^p}{\|m\|^p} d_{\mathcal{C}}(\pi^{(n)}(V),\pi^{(n)}(W)) \leq p.p! \dfrac{K^p}{\|m\|^p} \eta^{n}\Delta.$$
\end{proof}

The next two lemmas give ways to estimate $\|\widehat{M}\|_{\wedge,p}$.

\begin{lemma}
\label{normetenseur}
Let $M \in \mathcal{L}(E)$ and consider the operator norm of $\widehat{M}$ on $\bigwedge\nolimits^pE$ given by \linebreak $\|\widehat{M}\|_{\wedge,p}:=\underset{\|x\|_{\wedge,p}=1}{\sup} \|\widehat{M}(x) \|_{\wedge,p}$. Then this supremum is given by its value on the set of decomposable tensors, i.e.
$$\|\widehat{M}\|_{\wedge,p}=\underset{\|x_1\wedge \cdots \wedge x_p\|_{\wedge,p}=1}{\sup} \|\widehat{M}(x_1\wedge \cdots \wedge x_p) \|_{\wedge,p}.$$

%For $\|.\|_{\wedge,pr}$, $\|\widehat{M}\|\leq \sqrt{p!} \|\widehat{M}\|_0$.
\end{lemma}

\begin{proof}
See Appendix.
\end{proof}

\begin{prop}
\label{normeproduit}
Let $M \in \mathcal{L}(E)$. We have
$$\|\widehat{M}\|_{\wedge,p} \leq \|\bigwedge\nolimits^{p-1}M\|_{\wedge,p-1} .\|M\| \leq \|M\|^p$$

Similarly, for a continuous linear map $m : E \rightarrow \mathbb{C}^p$, we have $\|\hat{m}\| \leq \|m\|^p$.
\end{prop}

\begin{proof}
See Appendix.
\end{proof}

The next two lemmas allow us to use the fact that some space $W$ is in $\mathcal{C}[\rho]$ to get estimates on what happens when we move around it. This is then used to prove the estimates on $\|\widehat{M}\|_{\wedge,p}$ in Lemma \ref{encadrementnormeM}. 

Lemma \ref{spanboule1} asserts that under some conditions, if we can move around a $p$-space $W \in \mathcal{C}$ following directions $y_1,\dots,y_p$ for some time $r>0$ and stay in $\mathcal{C}$, then $\| y_1 \wedge \cdots \wedge y_p \|$ can't be too large and is controlled by $r$.

\begin{lemma}
\label{spanboule1}
Let $r>0$, $x_1,\cdots,x_p,y_1,\cdots,y_p \in E$ such that $\Span(x_1+t_1y_1,\cdots,x_p+t_py_p) \subset \mathcal{C}$ for all $(t_1,\cdots,t_p) \in \mathbb{C}^p$ with $|t_i| \leq r$, and such that $m(y_i) \in \Span(m(x_1),\cdots,m(x_i))$ for all $i \in \{1,\cdots,p\}$. Then 
$$\| y_1 \wedge \cdots \wedge y_p \| \leq \frac{K^p}{r^p}\|x_1\wedge \cdots \wedge x_p\|.$$
\end{lemma}

\begin{proof}
We first observe that given $a_1,\cdots,a_p,b_1,\cdots,b_p \in E$, we have 
$$\underset{\varepsilon_1,\cdots,\varepsilon_p \in \{-1,1\}}{\sum}(a_1+\varepsilon_1b_1) \wedge \cdots \wedge (a_p + \varepsilon_p b_p) = 2^p a_1 \wedge \cdots \wedge a_p.$$
Then 

$$\begin{array}{rll}
2^p r^p \| y_1 \wedge \cdots \wedge y_p \| &= 2^p  \| ry_1 \wedge \cdots \wedge ry_p \| \\
&\leq \underset{\varepsilon_1,\cdots,\varepsilon_p \in \{-1,1\}}{\sum}\|(ry_1+\varepsilon_1x_1) \wedge \cdots \wedge (ry_p + \varepsilon_p x_p) \|  \\
&\leq \underset{\varepsilon_1,\cdots,\varepsilon_p \in \{-1,1\}}{\sum}\|(x_1+\varepsilon_1ry_1) \wedge \cdots \wedge (x_p + \varepsilon_p ry_p) \| \\
&\leq \frac{K^p}{\|m\|^p} \underset{\varepsilon_1,\cdots,\varepsilon_p \in \{-1,1\}}{\sum} |\widehat{m}((x_1+\varepsilon_1ry_1) \wedge \cdots \wedge (x_p + \varepsilon_p ry_p) )| \\
\end{array}$$

Por $i=1..p$, let's set $f_i=m(x_i) \in \mathbb{C}^p$. Possibly after multiplicating the $x_i$ by complex phases we may assume that $\widehat{m}(x_1\wedge \cdots \wedge x_p)=\det(f_1,\cdots,f_p) > 0$. Possibly after multiplicating the $y_i$ by complex phases we may assume that the matrix of $(m(y_1),\cdots, m(y_n))$ in this basis $(f_1,\cdots, f_n)$ is of the form 
$$ \begin{pmatrix}
a_{11} & * &* \\
 0 &\ddots & *\\
 0 &0 & a_{pp} 
\end{pmatrix}$$

with $a_{ii} > 0$.
Then for $(t_1,\cdots,t_p) \in [-r,r]^p$, $\widehat{m}((x_1+t_1y_1) \wedge \cdots \wedge (x_p+t_py_p))= \widehat{m}(x_1\wedge \cdots \wedge x_p) . \underset{1\leq i\leq p}{\Pi} (1+t_ia_{ii})$ is in $\mathbb{R}$. As this quantity is never equal to $0$, we even get $\widehat{m}((x_1+t_1y_1) \wedge \cdots \wedge (x_p+t_py_p)) \in \mathbb{R}_{+}$. We finally get 

$$\begin{array}{rll}
2^p r^p \| y_1 \wedge \cdots \wedge y_p \| &\leq \frac{K^p}{\|m\|^p} \underset{\varepsilon_1,\cdots,\varepsilon_p \in \{-1,1\}}{\sum} |\widehat{m}((x_1+\varepsilon_1 r y_1) \wedge \cdots \wedge (x_p + \varepsilon_p ry_p) )| \\
&= \frac{K^p}{\|m\|^p} \underset{\varepsilon_1,\cdots,\varepsilon_p \in \{-1,1\}}{\sum} \widehat{m}((x_1+\varepsilon_1ry_1) \wedge \cdots \wedge (x_p + \varepsilon_p ry_p) ) \\

&= \frac{K^p}{\|m\|^p} \widehat{m}\left(\underset{\varepsilon_1,\cdots,\varepsilon_p \in \{-1,1\}}{\sum} (x_1+\varepsilon_1ry_1) \wedge \cdots \wedge (x_p + \varepsilon_p ry_p) \right) \\

&= \frac{K^p}{\|m\|^p} \widehat{m}\left( 2^p x_1 \wedge \cdots \wedge x_p \right) \\

&\leq 2^p K^p \|x_1 \wedge \cdots \wedge x_p \|.
\end{array}$$
Hence $\| y_1 \wedge \cdots \wedge y_p \| \leq \frac{K^p}{r^p}\|x_1\wedge \cdots \wedge x_p\|$.
\end{proof}

In the next lemma, given a $p$-space $W \in \mathcal{C}[\rho]$, we find  $x=x_1 \wedge \cdots \wedge x_p$ representing $W$ and $y=y_1 \wedge \cdots \wedge y_p$ verifying properties close to the ones from the previous lemma.

\begin{lemma}
\label{spanboule2}
Let $W \in \mathcal{C}[\rho]$, $M \in \mathscr{M}$. Let $x \in \bigwedge\nolimits^pE$ representing $W$ and $y \in \bigwedge\nolimits^pE$ be a decomposable tensor. Then there exists decompositions $x=x_1 \wedge \cdots \wedge x_p$ and $y=y_1 \wedge \cdots \wedge y_p$ verifying the following properties (compare with Lemma \ref{spanboule1}) with $r= \dfrac{\rho}{p2^{p-1}} \left( \dfrac{\|x\|}{p!\|y\|} \right)^{\frac{1}{p}}$, i.e.:
\begin{itemize}
\item $m(My_i) \in \Span(m(Mx_1),\cdots,m(Mx_i))$ for all $i \in \{1,\cdots,p\}$ ;
\item for all $(t_1,\cdots,t_p) \in \mathbb{C}^p$ with $|t_i| \leq r$, $\Span(x_1 +t_1y_1,\cdots, x_p+t_py_p) \subset \mathcal{C}$.
\end{itemize}
\end{lemma}

\begin{proof}
Let $x=x_1 \wedge \cdots \wedge x_p$ be an adapted decomposition of $x$ with $\|x_1\|=\cdots=\|x_p\|$. Let's denote $(f_1,\cdots,f_p):=(m(Mx_1),\cdots,m(Mx_p))$ ; it's a basis of $\mathbb{C}^p$ as $W \subset \mathcal{C}$. If $y=0$ then the decomposition $y=0\wedge \cdots \wedge 0$ works. Else, let's consider $V$ the $p$-vector space represented by $y$, and let's build an adapted decomposition $(y_1,\cdots,y_p)$ of $V$ such that $m(My_i) \in \Span(f_1,\cdots,f_i)$ for all $i \in \{1,\cdots,p\}$.

If $y_1,\cdots,y_i$ are built with linear forms $(l_1,\cdots,l_i)$ on $V$, we consider the restriction $m_i$ of $m  \circ M$ to $V \cap \ker l_1 \cap \cdots \cap \ker l_i$ : $$m_i : V \cap \ker l_1 \cap \cdots \cap \ker l_i \rightarrow \Span(f_1,\cdots,f_p).$$
If $m_i$ is not injective, then we take $y_{i+1} \in \ker m_i \setminus\{0\}$ and $l_{i+1}$ a linear form associated to $y_{i+1}$. If $m_i$ is injective, then $\dim m_i(V \cap \ker l_1 \cap \cdots \cap \ker l_i) \geq p - i$ hence $m_i(V \cap \ker l_1 \cap \cdots \cap \ker l_i) \cap \Span(f_1,\cdots,f_{i+1}) \neq \{0\}$. Thus we can take $y_{i+1} \neq 0$ in $V \cap \ker l_1 \cap \cdots \cap \ker l_i$ such that $m\circ M (y_{i+1}) \in \Span(f_1,\cdots,f_{i+1})$, and $l_{i+1}$ a linear form associated to $y_{i+1}$.

We can now write $y=y_1 \wedge \cdots \wedge y_p$ with $(y_1,\cdots,y_p)$ an adapted decomposition of $V$ such that $m(y_i) \in \Span(f_1,\cdots,f_i)$ for all $i \in \{1,\cdots,p\}$ and $\|y_1\| = \cdots = \|y_p\|$. Let $(t_1,\cdots,t_p) \in \mathbb{C}^p$ with $|t_i| \leq r$ for all $i$, and $u=\underset{i}{\sum} \lambda_i (x_i + t_i y_i) \in \Span(x_1 +t_1y_1,\cdots, x_p+t_py_p)$. We want to show that $u \in \mathcal{C}$, by using the fact that $\Span(x_1,\cdots, x_p) \subset \mathcal{C}[\rho]$. As $|\lambda_i|.\|x_i\| \leq 2^{i-1} \|\underset{j}{\sum} \lambda_j x_j\|$ by Lemma \ref{estimation-coordonnees}, we get 
\begin{align}
\label{eq}
\|\underset{i}{\sum} \lambda_i t_i y_i \| \leq \underset{i}{\sum} |\lambda_i|.| t_i|.\| y_i\|
\leq  \underset{i}{\sum} 2^{i-1}.| t_i|.\dfrac{\| y_i\|}{\|x_i\|}.\|\underset{j}{\sum} \lambda_j x_j\|.
\end{align}

As $\|x_1\|=\cdots=\|x_p\|$, $\|y_1\|=\cdots=\|y_p\|$ and these are adapted decompositions, we get 
$$\|y_i\|^p=\|y_1\|\cdots\|y_p\| \leq \|y\| \leq \left( \dfrac{\rho}{rp2^{p-1}} \right)^p \dfrac{\|x\|}{p!} \leq \left( \dfrac{\rho}{rp2^{p-1}} \right)^p \|x_1\|\cdots\|x_p\| \leq \left( \dfrac{\rho}{rp2^{p-1}} \|x_i\|\right)^p $$
Finally, inequality (\ref{eq}) becomes
$$\begin{array}{rll}
\|\underset{i}{\sum} \lambda_i t_i y_i \| &\leq  \underset{i}{\sum} 2^{i-1}.\dfrac{| t_i|}{r}.\dfrac{\rho}{p2^{p-1}}.\|\underset{j}{\sum} \lambda_j x_j\| \leq \rho.\|\underset{j}{\sum} \lambda_j x_j\|.
\end{array}
$$

Using $\underset{j}{\sum} \lambda_j x_j \in \mathcal{C}[\rho]$, we therefore get $u=\underset{i}{\sum} \lambda_i (x_i + t_i y_i) \in \mathcal{C}$, hence $\Span(x_1 +t_1y_1,\cdots, x_p+t_py_p) \subset \mathcal{C}$, which achieves the proof.

\end{proof}

We can now use the two previous lemmas to get the following:

\begin{lemma}
\label{encadrementnormeM}
Let $M \in \mathscr{M}$, $W \in \mathcal{C}[\rho]$ and $x \in \bigwedge\nolimits^p E$ representing $W$. Then
$$ \dfrac{1}{K^p} \left| \dfrac{\widehat{m}(\widehat{M}x)}{\widehat{m}(x)} \right| \leq \|\widehat{M}\| \leq p! \left( \dfrac{2^{p-1}K^2}{\rho^p} \right)^p \left| \dfrac{\widehat{m}(\widehat{M}x)}{\widehat{m}(x)} \right|$$
\end{lemma}

\begin{proof}
Let $y \in \bigwedge\nolimits^pE$ be a decomposable tensor. By Lemma~\ref{spanboule2}, we get decompositions \linebreak$x=x_1 \wedge \cdots \wedge x_p$ and $y=y_1 \wedge \cdots \wedge y_p$ such that the decompositions $\widehat{M}x= Mx_1 \wedge \cdots \wedge Mx_p$ and $\widehat{M}y=My_1 \wedge \cdots \wedge My_p$ satisfy the hypothesis of Lemma~\ref{spanboule1} for $r=\frac{\rho}{p2^{p-1}} \left( \frac{\|x\|}{p!\|y\|} \right)^{\frac{1}{p}}$. This gives

$$\| \widehat{M}y \| \leq \dfrac{p!(2^{p-1}K)^p}{\rho^p}\dfrac{\|y\|}{\|x\|}.\|\widehat{M}x\| \leq \dfrac{p!(2^{p-1}K)^p}{\rho^p} .K^p.\left|\dfrac{\widehat{m}(\widehat{M}x)}{\widehat{m}(x)}\right|.\|y\|$$

where the last inequality commes from the properties of $\widehat{m}$. By Lemma~\ref{normetenseur}, $\|\widehat{M}\|$ is given by its supremum over the decomposable tensors, hence $\|\widehat{M}\| \leq p! \left( \frac{2^{p-1}K^2}{\rho^p} \right)^p \left| \frac{\widehat{m}(\widehat{M}x)}{\widehat{m}(x)} \right|$. The properties of $\widehat{m}$ also give $|\widehat{m}(\widehat{M}x) |\leq \|\widehat{M}\|.\|\widehat{m}\|.\|x\| \leq  \|\widehat{M}\|.K^p |\widehat{m}(x)|$.
\end{proof}
~ \\ ~ \\

Before going on in our proof, we need to define $\mathscr{E}(\Omega,G_p(\mathcal{C}))$ (the set of mesurable maps from $\Omega$ to $G_p(\mathcal{C})$) and what it means to be analytic on this space. As neither $G_p(E)$ or $\bigwedge\nolimits^pE$ are Banach spaces, the analyticity has to come from the manifold structure, i.e. from the charts on $G_p(E)$. Fortunately in our case, when studying $G_p(\mathcal{C})$, we can restrict the space we look at to a single chart. 

We have something even better: there exists a chart of $G_p(E)$ which covers $G_p(\mathcal{C})$ and in which the image of $G_p(\mathcal{C})$ is bounded. Indeed, let $V_0=\ker m$ and $W_0 \in G_p(\mathcal{C})$. Then $W_0 \cap V_0 = \{0\}$ and given the dimension and co-dimension, we have $E=W_0 \oplus V_0$ (and they are topological supplements as they are closed). For $W \in  G_p(\mathcal{C})$, $W$ is a supplement of $V_0$, the chart $A \in \mathcal{L}(W_0,V_0) \mapsto \graph A \in G_p(E)$ covers $G_p(\mathcal{C})$. Moreover, if $\graph A \in G_p(\mathcal{C})$, then for $x \in W_0$ we have $x+Ax \in \mathcal{C}$ and
$$\|Ax\| \leq \|x+Ax\| + \|x\| \leq |m(x+Ax)| +\|x\| = |m(x)| +\|x\| \leq (K+1)\|x\|.$$

Therefore the image of $G_p(\mathcal{C})$ in the chart $\mathcal{L}(W_0,V_0)$ is bounded by $(K+1)$. Now let $\mathscr{E}(\Omega,G_p(\mathcal{C}))$ be the set of measurable maps from $\Omega$ to $G_p(\mathcal{C})$. Then by identifying $G_p(\mathcal{C})$ and $\mathcal{L}(W_0,V_0)$, we can consider $\mathscr{E}(\Omega,G_p(\mathcal{C}))$ as a subset of $\mathscr{B}(\Omega, \mathcal{L}(W_0,V_0))$, thus giving it an analytical Banach space structure. We define the structure of $\mathscr{E}(\Omega,\widehat{\mathcal{C}}_{m=1})$ by the identification $\widehat{\mathcal{C}}_{m=1} \simeq G_p(\mathcal{C})$.

We can now define our objects properly and formulate a fixed point theorem.
\begin{defi}
\label{defipit}
For $t \in \mathbb{D}$, we define:
$$\pi_t :\left\{ \begin{array}{ccc}
\mathscr{E}(\Omega,G_p(\mathcal{C})) &\rightarrow &\mathscr{E}(\Omega,G_p(\mathcal{C})) \\
(V_{\omega})_{\omega\in\Omega} &\mapsto &\pi_t(V): \omega \mapsto \pi_{M_{\omega}(t)}(V_{\tau \omega})= M_{\omega}(t) (V_{\tau \omega})
\end{array} \right.$$

Or equivalently,

$$\pi_t :\left\{ \begin{array}{ccc}
"\mathscr{E}(\Omega,\widehat{\mathcal{C}}_{m=1})" &\rightarrow &"\mathscr{E}(\Omega,\widehat{\mathcal{C}}_{m=1})" \\
(v_{\omega})_{\omega\in\Omega} &\mapsto &\pi_t(v): \omega \mapsto \pi_{M_{\omega}(t)}(v_{\tau \omega})= \dfrac{\widehat{M_{\omega}}(t) (v_{\tau \omega})}{\widehat{m}(\widehat{M_{\omega}}(t) (v_{\tau \omega}))}
\end{array} \right.$$

\end{defi}

\begin{lemma}
\label{pointfixe}
For $t \in \mathbb{D}$, $\pi_t$ has a unique fixed point $V^*(t)$ in $\mathscr{E}(\Omega,G_p(\mathcal{C}))$ - or equivalently, $v^*(t)$ in $\mathscr{E}(\Omega,\widehat{\mathcal{C}}_{m=1})$.
\end{lemma}

\begin{proof}
Let $t\in \mathbb{D}$. By Lemma \ref{contractionpim}, there exists $C>0$ such that $\textnormal{diam}( \pi_t^{n}(\mathscr{E}(\Omega,G_p(\mathcal{C}))) \leq C\eta^n$. Pick $V_0 \in G_p(\mathcal{C})$ and define $V^0(t)_\omega = V_0$ for all $\omega \in \Omega$. Then the sequence defined by $V^{n+1}(t)=\pi^t(V^n(t))$ is Cauchy. As $\mathcal{C}$ is closed, $G_p(\mathcal{C})$ is a complete metric space and therefore $V^n(t)$ converges to some $V^*(t) \in \mathscr{E}(\Omega,G_p(\mathcal{C}))$. As $\pi_t$ is continuous, $V^*(t)$ is a fixed point and the uniqueness follows from Lemma~\ref{contractionpim}.
\end{proof}

We will now use this fixed point and the previous lemmas to compute $\chi_p(t)$. Then we will show that $t  \mapsto v^*(t)$ is analytic and then conclude with the real-harmonicity of $t \mapsto \chi_p(t)$.

\begin{defi}
\label{defi p}
We define the map $$p : \left\{ \begin{array}{ccl}
\mathbb{D} &\rightarrow &\mathscr{E}(\Omega,\mathbb{C}) \\
t &\mapsto	&\dfrac{\widehat{m}(\widehat{M}_\omega(t)v^*_{\tau \omega}(t))}{\widehat{m}(v^*_{ \omega}(t))} =\widehat  {m}(\widehat{M}_\omega(t)v^*_{\tau \omega}(t))
\end{array} \right.$$
\end{defi}

\begin{lemma}
\label{chi int}
For $t \in \mathbb{D}$, $\chi_p(t)=\int \log |p_{\omega}(t)| d\mu(\omega)$.	
\end{lemma}

\begin{proof}
As $v^*(t)$ is a fixed point of $\pi^y$, $v^*(t) \in M(\mathcal{C}) \subset \mathcal{C}[\rho]$. By Lemma \ref{encadrementnormeM}, there exists $C1,C2 >0$ such that for all $\omega \in \Omega, t \in \mathbb{D}$,
$$ C_1 \left| \widehat{m}\left(\widehat{M_\omega(t)}v^*_{\tau \omega}(t)\right) \right| \leq \|\widehat{M}_\omega (t)\| \leq C_2 \left| \widehat{m}\left(\widehat{M_\omega(t)}v^*_{\tau \omega}(t)\right) \right|.$$

This means that $|p_{\omega}(t)|$ is equivalent to $\|\widehat{M}_\omega (t)\|\leq \|M_\omega (t)\|^p$. By Assumption 2 of Theorem \ref{theorem produit}, this is dominated by $\|M_\omega (0)\|^p$, which is log-integrable by Assumption 3 of the same theorem. This proves $(\omega \mapsto \log |p_{\omega}(t)|) \in L^1(\Omega,\mu)$. Applying Lemma \ref{encadrementnormeM} to $M^{(n)}_\omega(t)$ and $v^*_{\tau^n \omega}(t)$, we get
$$ C_1 \left| \widehat{m}\left(\widehat{M^{(n)}_\omega(t)}v^*_{\tau^n \omega}(t)\right) \right| \leq \|\widehat{M^{(n)}}_\omega (t)\| \leq C_2 \left| \widehat{m}\left(\widehat{M^{(n)}_\omega(t)}v^*_{\tau^n \omega}(t)\right) \right|.$$
Then
$$\begin{array}{lll}
\dfrac{1}{n} \log \left\| \widehat{M_\omega^{(n)}(t)} \right\| &= &\dfrac{1}{n} \log \left| {\widehat{m}(\widehat{M}_\omega^{(n)}(t)v^*_{\tau^n \omega}(t))} \right| +O\left( \dfrac{1}{n}\right) \\
&=&\dfrac{1}{n} \log  \left|\underset{k=0}{\overset{n-1}{\Pi}} \widehat{m}(\widehat{M}_\omega^{(k+1)}(t)v^*_{\tau^{k+1} \omega}(t)) \right| +O\left( \dfrac{1}{n}\right) \\
&= &\dfrac{1}{n} \underset{k=0}{\overset{n-1}{\sum}} \log |p_{\tau^k \omega}(t)| + O\left( \dfrac{1}{n}\right).
\end{array}$$

As $(\omega \mapsto \log |p_{\omega}(t)|) \in L^1(\Omega,\mu)$ and $\tau$ is ergodic, Birkhoff's ergodic theorem gives \linebreak $\dfrac{1}{n} \log \left\| \widehat{M_\omega^{(n)}(t)} \right\| \tendn \int \log |p_{\omega}(t)| d\mu(\omega)$, which is the desired result.
\end{proof}

\begin{lemma}
\label{pim differentiable}

Let $W \in \mathcal{C}[\rho]$ and $x = G^ {-1}(G_p(\mathcal{C})) \in \bigwedge\nolimits^pE$ representing $W$. Let $h \in \bigwedge\nolimits^pE$ such that $x+h \in G^ {-1}(G_p(\mathcal{C}))$. Then
$$\|\pi^{(n)}(x+h) - \pi^{(n)}(x) \| \leq C.\eta^n \|h\| + o(\|h\|) $$
as $h \rightarrow 0$, where $C$ is a constant depending only on $p$, $K$, $\|m\|$ and $\rho$.
\end{lemma}

\begin{proof}
Let $W$ and $W_h$ be the vector spaces represented by $x$ and $x+h$. Then
$$\begin{array}{rcll}
\|\pi^{(n)}(x+h) - \pi^{(n)}(x) \| &= &\left\|\frac{\widehat{M}^{(n)}(x+h)}{\widehat{m}(\widehat{M}^{(n)}(x+h))} - \frac{\widehat{M}^{(n)}x}{\widehat{m}(\widehat{M}^{(n)}x)} \right\| \\
&\leq & p! \frac{K^p}{\|m\|^p}d_\mathcal{C}(M^{(n)}W,M^{(n)}W_h) \\
&\leq & p! \frac{K^p}{\|m\|^p} \eta^n d_\mathcal{C}(W,W_h) &\textnormal{by Proposition \ref{contraction}}\\
&\leq &p! \frac{K^p}{\|m\|^p}C_1 \eta^nd_{Gr}(W,W_h) &\textnormal{where $C_1$ is given by Lemma \ref{aperture-2}}\\
&\leq & p! \frac{K^p}{\|m\|^p}C_1 C_2 \eta^nd_{\wedge}(W,W_h) &\textnormal{where $C_2$ is given by Theorem \ref{metriques-equivalentes}}\\
&\leq & p! \frac{K^p}{\|m\|^p}C_1 C_2 \eta^n\|\frac{x+h}{\|x+h\|}-\frac{x}{\|x\|} \| \\
&\leq & C \eta^n(\|h\|+o(\|h\|)
\end{array}$$
\end{proof}

\begin{lemma}
\label{v analytic}
The map $t \in \mathbb{D} \mapsto v^*(t) \in "\mathscr{E}(\Omega,G_p(\mathcal{C}))"$ is analytic.
\end{lemma}

\begin{proof}

\begin{comment}

Let $t_0 \in \mathbb{D}$. Let $C_1,C_2,C_3$ be the constants given by Lemma \ref{pim-controle}. Using ASSUMPTION 2 , we can find $\delta >0$ such that 
$$\forall t \in B(t_0,\delta),\forall \omega \in \Omega, \| M_\omega(t)-M_\omega(t_0)\| \leq C_1$$
\end{comment}

If $0<\rho' < \rho$, $\mathcal{C}[\rho']$ is a neighborhood of $\mathcal{C}[\rho]$ in $G_p(E)$. Indeed, let $W \subset \mathcal{C}[\rho]$ and let $W' \in G_p(E)$ such that $d_H(W,W')< d$ ($d$ will be fixed later). Let $x' \in W'$ and $y' \in E$ such that $\|y'\| \leq \rho' \|x'\|$. As $d_H(W,W')< d$, there exists $x \in W$ such that $\|x-x'\|<d\|x\|$. Then
$$\|(x'+y')-x\| \leq \|x'-x\|+\|y'\| \leq d\|x\|+\rho' \|x'\| \leq d\|x\|+\rho' (d+1)\|x\| \leq \left(d+\rho' (d+1)\right)\|x\|.$$
Taking $d=\frac{\rho -\rho'}{2(1+\rho')}$, we get $d+\rho' (d+1)=\frac{\rho+\rho'}{2} < \rho$ which gives $\|(x'+y')-x\| \leq \rho\|x\|$ and therefore $x'+y' \in \mathcal{C}$. This shows that $x' \in  \mathcal{C}[\rho']$ hence $W' \subset \mathcal{C}[\rho']$. We proved $B_{d_H}(W,d) \subset \mathcal{C}[\rho']$ for all $W \subset \mathcal{C}[\rho]$, therefore $\mathcal{C}[\rho']$ is a neighborhood of $\mathcal{C}[\rho]$ in $G_p(E)$.

%Then we can show that $B(\Omega,G^{-1}(\mathcal{C}[\rho']))$ is a neighborhood of $B(\Omega,G^{-1}(\mathcal{C}[\rho]))$ in $B(\Omega,G^{-1}(\mathcal{C}))$. Let $v \in B(\Omega,G^{-1}(\mathcal{C}[\rho]))$, let $L$ be the Lipschitz coefficient of $G$ (given in the proof of Lemma \ref{homeo grassmannienne}) and let $d=\frac{\rho -\rho'}{2(1+\rho')}$. Let $v' \in B(\Omega,G^{-1}(\mathcal{C})$ such that $\|v'_\omega-v\omega\| < L^{-1}d$ for almost every $\omega \in Omega$. Then $d_H(G(v'_\omega),G(v_\omega)) < d$, which gives $G(v'_\omega) \in \mathcal{C}[\rho']$, i.e. $v'_\omega \in B(\Omega,G^{-1}(\mathcal{C}[\rho']))$.

Consider the map

$$\Phi : \begin{array}{ccc}
\mathbb{D}\times "\mathscr{E}(\Omega,G_p(\mathcal{C}[\rho/2]))" &\longrightarrow & \mathscr{E}"(\Omega,G_p(\mathcal{C}))" \\
(t,v:\omega\mapsto v_\omega) &\mapsto &\pi_t(v)
\end{array}$$

%As $v_{\tau\omega} \in \mathcal{C}[\rho/2]$, by Lemma \ref{encadrementnormeM},
%$$\begin{array}{rcl}
%\|\pi_t(v)_\omega\|&= &\left\|\dfrac{\widehat{M_{\omega}}(t) (v_{\tau \omega})}{\widehat{m}(\widehat{M_{\omega}}(t) (v_{\tau \omega}))}\right\| \\
%&\leq & \dfrac{\|\widehat{M_{\omega}}(t) (v_{\tau \omega})\|}{|\widehat{m}(v_{\tau \omega})| \|\widehat{M}_\omega(t)\|} \\
%&\leq &\dfrac{\|\widehat{M_{\omega}}(t) (v_{\tau \omega})\|}{\|\widehat{M}_\omega(t)\|}
%\end{array}$$

To see if this map is analytic, we remember that there is a chart of $G_p(E)$ which covers $G_p(\mathcal{C})$ and in which the image of $G_p(\mathcal{C})$is bounded: the image of $G_p(\mathcal{C})$ in some chart $\mathcal{L}(W_0,V_0)$ is bounded by $(K+1)$. In this chart, we can write $\phi$ as 
$$\tilde{\phi} : \begin{array}{ccc}
\mathbb{D}\times B\left(\mathscr{B}(\Omega,\mathcal{L}(W_0,V_0)),K+1\right) &\longrightarrow & \mathscr{B}(\Omega,\mathcal{L}(W_0,V_0)) \\
(t,A:\omega\mapsto A_\omega) &\mapsto &\tilde{\pi_t}(A)
\end{array}$$

We note $A^*(t)$ the element of $B\left(\mathscr{B}(\Omega,\mathcal{L}(W_0,V_0)),K+1\right)$ associated to the fixed point $v^*(t)$.
 Let $t_0 \in \mathbb{D}$ and let's denote $T_0=D_A\tilde{\pi_{t_0}}(A^*(t_0))$ the derivative of $\tilde{pi_{t_0}}$ at the fixed point $A^*(t_0)$. By Lemma \ref{pim differentiable}, $T_0^n=D_A\tilde{\pi_{t_0}^{(n)}}(A^*(t_0))$ verifies $\|T_0^n\|\leq C \eta^n$ for some constant $C$. Then the spectral radius of $T_0$ is less that $\eta$ and therefore the derivative $Id - T_0$ of $v \mapsto A - \tilde{\pi_t}(A)$ is invertible at the fixed point $A^*(t_0)$.
 
We can now apply the implicit function theorem and conclude there is an analytic function $t \mapsto A^*(t)$ defined on a neighborhood of $t_0$ for which $A^*(t) - \tilde{\pi_t}(A^*(t))=0$. This proves the analyticity of $t \mapsto A^*(t)$ and therefore the analyticity of $t \mapsto v^*(t)$.

\end{proof} 

\begin{lemma}
\label{derivee chapeau}
For $t \in \mathbb{D}$,
$$\left\|\dfrac{d}{dt} \widehat{M}_\omega(t) \right\| \leq p \left\| \bigwedge\nolimits^{p-1}M_\omega(t) \right\| \left\|\dfrac{d}{dt} M_\omega(t) \right\|$$
Under the assumptions of Theorem \ref{theorem produit}, this gives
$$\left\|\dfrac{d}{dt} \widehat{M}_\omega(t) \right\| \leq p C \left\|\widehat{M}_\omega(t) \right\|$$
where $C = \sup \left\{ \dfrac{\|\bigwedge\nolimits^{p-1}M_\omega(t)\|}{\|\bigwedge\nolimits^{p}M_\omega(t)\|}\|\frac{d}{dt}M_\omega(t)\| : \omega \in \Omega, t \in \mathbb{D}\right\}$.
\end{lemma}

\begin{proof}
See Appendix.
\end{proof}

\begin{proof}[Proof Theorem \ref{theorem produit}]

Let $\omega \in \Omega$. The map $t  \mapsto p_\omega(t)$ is analytic, but what we truly desire is the analyticity of $t  \mapsto \log p_\omega(t)$, which requires to define a complex logarithm in a consistent way. Let $\epsilon >0$ and $t_0 \in \mathbb{D}$. As $t \mapsto v^*(t)$ is analytic, there exist $\delta >0$ such that  $\|v^*_{\tau\omega}(t)-v^*_{\tau\omega}(t_0)\| \leq \epsilon$ for $|t-t_0|< \delta$. By Lemma \ref{derivee chapeau}, we can choose $\delta$ small enough such that $\|\widehat{M}_\omega(t)-\widehat{M}_\omega(t_0)\| \leq \epsilon\|\widehat{M}_{t_0}\|$. Then for $|t-t_0|< \delta$, we get

$$\begin{array}{rcl}
|p_\omega(t)-p_\omega(t_0)| &=& |\widehat{m}(\widehat{M}_\omega(t)v^*_{\tau\omega}(t)-\widehat{M}_\omega(t_0)v^*_{\tau\omega}(t_0))| \\
&\leq & K^p\|(\widehat{M}_\omega(t)v^*_{\tau\omega}(t)-\widehat{M}_\omega(t_0)v^*_{\tau\omega}(t_0)\| \\
&\leq & K^p\|\widehat{M}_\omega(t)-\widehat{M}_\omega(t_0)\| \|v^*_{\tau\omega}(t)\| +K^p\|\widehat{M}_\omega(t_0)\| \|v^*_{\tau\omega}(t)-v^*_{\tau\omega}(t_0)\| \\
&\leq & K^p \epsilon + K^p\|\widehat{M}_\omega(t_0)\| \|v^*_{\tau\omega}(t)-v^*_{\tau\omega}(t_0)\| \\
&\leq & 2K^p \epsilon \|\widehat{M}_\omega(t_0)\| \\
\end{array}$$

By Lemma \ref{encadrementnormeM},  $|p_\omega(t_0)| \geq C_1 \|\widehat{M}_\omega(t_0)\|$ for some constant $C_1 >0$, hence
$$\left|\dfrac{p_\omega(t)}{p_\omega(t_0)}-1 \right| \leq 2 K^p \epsilon /C_1.$$

Fixing $\epsilon$ small enough such that $2 K^p \epsilon /C_1 < 1/2$, we have $\left|\dfrac{p_\omega(t)}{p_\omega(t_0)}-1 \right| \leq 1/2$. Then, for $|t-t_0| \leq \delta$,
$$\chi(t)-\chi(t_0) = \int \log \left|\dfrac{p_\omega(t)}{p_\omega(t_0)}\right| d\mu(\omega) = \textnormal{Re} \int \log \dfrac{p_\omega(t)}{p_\omega(t_0)} d\mu(\omega)$$
where $\log$ is the usual logarithm on $\mathbb{C}\setminus \mathbb{R}_-$. This shows that $t \mapsto \chi(t)$ is real-analytic.
\end{proof}

\begin{ex}
Let $(\xi_n)_{n \geq 0}$ be a sequence of independent and identically distributed random variables with values in $\mathbb{D}$. Consider for $t \in \mathbb{D}$ 
$$M_n(t) = \begin{pmatrix}
10+t\xi_n &t+\xi_n & it \\
t+\xi_n & 6 &\xi_n \\
i\xi_n &0 &1
\end{pmatrix}$$

With some computations, we can show that $M_n(t)$ maps $\mathcal{C}_{\pi,1}$ into $\mathcal{C}_{\pi,0.85}$, where $\pi$ is the orthogonal projection on the first two coordinates. Conditions $1$ and $3$ of Theorem \ref{theorem produit} are obviously satisfied, and condition $2$ is also satisfied as $\det M_n(t) \geq 43$ and $\|M_n(t)\|$ and $\|\frac{d}{dt}M_n(t)\|$ are bounded. By Theorem \ref{theorem produit}, a.s. 
$$\chi_2(t)=\lim \dfrac{1}{n} \log \left\| (M_1(t) \wedge M_1(t))\cdots (M_n(t) \wedge M_n(t))\right\|$$
exists and defines a harmonic function of $t \in \mathbb{D}$. We can do a similar reasoning to show that $\chi_1(t)=\lim \dfrac{1}{n} \log \left\| M_1(t)\cdots  M_n(t)\right\|$ exists and defines a harmonic function of $t \in \mathbb{D}$. Thus the first two characteristic exponents of the random product of $t \mapsto (M_n(t))$ are harmonic. Moreover, by Birkhoff's theorem, we get $$\chi_3(t)=\lim \dfrac{1}{n} \log \det \left(M_1(t)\cdots  M_n(t)\right)= \mathbb{E}\left[\log \begin{vmatrix}
10+t\xi_0 &t+\xi_0 & it \\
t+\xi_0 & 6 &\xi_0 \\
i\xi_0 &0 &1
\end{vmatrix}\right],$$thus all three characteristic exponents of the random product of $t \mapsto (M_n(t))$ are harmonic.
\end{ex}

\section*{Appendix}

In this appendix we provide the proofs which only involve computations on $\bigwedge\nolimits^pE$.
\begin{customdefi}{\ref{defi1}}
If $l_1,\cdots,l_p \in E'$ and $x_1,\cdots,x_p \in E$, we define $$\langle l_1 \wedge \cdots \wedge l_p ,x_1 \wedge \cdots \wedge x_p \rangle = \det ((\langle l_i, x_j\rangle)_{i,j})$$
which can be extended to a linear form $l_1 \wedge \cdots \wedge l_p$ on $\bigwedge\nolimits^p E$. We define a norm $||.||_{\wedge 1, p}$ on $\bigwedge\nolimits^p E$ by

\[ ||x||_{\wedge1, p} = \underset{\begin{array}{c}l_1,\cdots,l_p \in E' \\ ||l_1||=\cdots =||l_p||=1 \end{array}}{\sup} \left| \langle l_1 \wedge \cdots \wedge l_p ,x \rangle \right| \]
\end{customdefi}

\begin{proof}
The homogeneity and the triangle inequality are obvious. Let $x \in \bigwedge\nolimits^p E$ such that $||x||_{\wedge1, p}=0$. Let $x= \sum\limits_i x_1^i \wedge \dots \wedge x_p^i$ be a decomposition of $x$. If we take $(f_1,\dots,f_k)$ a basis of Span$((x_j^i)_{i,j})$, then we can write $x= \sum\limits_{1\leq i_1< \dots < i_p \leq k} a_{i_1,\dots,i_p} f_{i_1} \wedge\dots	\wedge f_{i_p}.$
Taking for $(l_1,\dots,l_k)$ the dual basis of $(f_1,\dots,f_k)$ and extending these linear forms to $E$ with the Hahn-Banch theorem, we get $0=\langle l_{i_1} \wedge \cdots \wedge l_{i_p} ,x \rangle = a_{i_1,\dots,i_p}$, hence $x=0$. This proves the separation of the norm $||.||_{\wedge 1, p}$.
\end{proof}

\begin{customdefi}{\ref{defi2}}
We define a norm $||.||_{\wedge 2, p}$ on $\bigwedge\nolimits^p E$ by
\item \[ \|x\|_{\wedge2,p} = \inf \underset{i}{\sum} \|x^i_1\| \cdots \|x^i_p \| \]
where the infimum is taken over all decompositions $x = \underset{i}{\sum} x^i_1 \wedge \cdots \wedge x^i_p $.
\end{customdefi}

\begin{proof}
Let $x,y \in \bigwedge\nolimits^p E$ and $\lambda \in \mathbb{C}$. If $x = \underset{i}{\sum} x^i_1 \wedge \cdots \wedge x^i_p $ is a decomposition of $x$, then $\lambda x = \underset{i}{\sum} \lambda x^i_1 \wedge \cdots \wedge x^i_p $ is a decomposition of $\lambda x$. Taking the infimum over all decompositions of $x$, we get $\|\lambda x\|_{\wedge2,p} \leq |\lambda|  \|x\|_{\wedge2,p}$. If $\lambda=0$, then $\|0. x\|=0=0.\|x\|_{\wedge2,p}$. If $\lambda \neq 0$, then our reasoning gives  $\| x\|_{\wedge2,p} \leq \frac{1}{|\lambda|}  \|\lambda x\|_{\wedge2,p}$, hence $\|\lambda x\|_{\wedge2,p}=|\lambda|  \|x\|_{\wedge2,p}$. \\
If $x = \underset{i}{\sum} x^i_1 \wedge \cdots \wedge x^i_p $ and $y = \underset{j}{\sum} y^j_1 \wedge \cdots \wedge y^j_p $, then $x+y = \underset{i}{\sum} x^i_1 \wedge \cdots \wedge x^i_p + \underset{j}{\sum} y^j_1 \wedge \cdots \wedge y^j_p $ is a decomposition of $x+y$, thus $\underset{i}{\sum} \|x^i_1\| \cdots \|x^i_p \| + \underset{j}{\sum} \|y^j_1\| \cdots \|y^j_p \| \geq \|x+y\|_{\wedge2,p} $. Taking the infimum over all decompositions of $x$ and $y$, we get $\|x\|_{\wedge2,p} +\|y\|_{\wedge2,p}\geq \|x+y\|_{\wedge2,p} $. \\
The separation comes from the following proposition.
\end{proof}

\begin{customprop}{\ref{prop1}}
On $\bigwedge\nolimits^p E$, $\|.\|_{\wedge1,p} \leq (\sqrt{p})^p \|.\|_{\wedge2,p}$.
\end{customprop}

\begin{proof}
Let $x = \underset{i}{\sum} x^i_1 \wedge \cdots \wedge x^i_p$ and let $l_1, \cdots, l_p \in E'$ with $\|l_1\| = \cdots = \|l_p \| = 1$. Then 
$$\begin{array}{rcl}

|\langle l_1 \wedge \cdots \wedge l_p , x \rangle| &= &|\underset{i}{\sum} \langle l_1 \wedge \cdots \wedge l_p , x^i_1 \wedge \cdots \wedge x^i_p \rangle| \\
&\leq &\underset{i}{\sum} |\langle l_1 \wedge \cdots \wedge l_p , x^i_1 \wedge \cdots \wedge x^i_p \rangle| \\
&\leq &\underset{i}{\sum} \left|\det \begin{pmatrix}
\langle l_1, x^i_1 \rangle & \cdots & \langle l_p, x^i_1 \rangle\\
\vdots & &\vdots \\
\langle l_1, x^i_p\rangle &\cdots &\langle l_p, x^i_p\rangle
\end{pmatrix} \right|\\
&\leq &\underset{i}{\sum} (\sqrt{p})^p\|x^i_1 \|  \cdots \|x^i_p \|.
\end{array}$$ 
Taking the supremum over all $l_i$ and the infimum over all decompositions of $x$, we get $\|x\|_{\wedge1,p} \leq p! \|x\|_{\wedge2,p}$.
\end{proof}

Let's now look at some properties of these norms.

\begin{customlemma} {\ref{p->p+1}}
Let $x \in E$ and $u \in \bigwedge\nolimits^pE$. Then
\begin{enumerate}[(i)]
\item $|| x \wedge u ||_{\wedge1,p+1} \leq (p+1) ||x||.||u||_{\wedge1,p}$
\item $|| x \wedge u ||_{\wedge2,p+1} \leq ||x||.||u||_{\wedge2,p}$
\end{enumerate}
\end{customlemma}

\begin{proof}
\begin{enumerate}[(i)]
\item Let $x \in E$ and let $l_1,l_2, \cdots, l_{p+1}$ be independent linear forms of norm 1 on $E$. Let $\hat{l_i} = l_1 \wedge \cdots l_{i-1} \wedge l_{i+1} \wedge \cdots \wedge l_{p+1}$ and $f: u \in \bigwedge\nolimits^mE \mapsto \underset{i=1}{\overset{p+1}{\sum}} (-1)^{i+1}l_i(x)\hat{l_i}(u)$. As for all $(x_1,\cdots,x_p) \in E^p$ we have $ l_1 \wedge \cdots \wedge l_{p+1} (x \wedge x_1 \wedge \cdots \wedge x_p) = f(x_1 \wedge \cdots \wedge x_p)$, we have for all $u \in \bigwedge\nolimits^pE$ : 
\[  l_1 \wedge \cdots \wedge l_{p+1} (x \wedge u) = f(u) = \underset{i=1}{\overset{p+1}{\sum}} (-1)^{i+1}l_i(x)\hat{l_i}(u) \]
Hence $| l_1 \wedge \cdots \wedge l_{p+1} (x \wedge u)| \leq \underset{i=1}{\overset{p+1}{\sum}}|l_i(x)|.|\hat{l_i}(u)| \leq  (p+1) ||x||.||u||_{\wedge,p}$ and therefore $|| x \wedge u ||_{\wedge,p+1} \leq (p+1) ||x||.||u||_{\wedge,p}$.
\item Let $u=\underset{i}{\sum} u^i_1 \wedge \cdots \wedge u^i_p$. Then $x \wedge u =\underset{i}{\sum} x \wedge u^i_1 \wedge \cdots \wedge u^i_p$ is a decomposition of $x \wedge u$, hence $\|x \wedge u\|_{\wedge2,p+1} \leq \underset{i}{\sum} \|x\|.\| u^i_1 \|\cdots \| u^i_p\|$. Taking the infimum over all decompositions of $u$, we get $|| x \wedge u ||_{\wedge2,p+1} \leq ||x||.||u||_{\wedge2,p}$.
\end{enumerate}

\end{proof}

\begin{customlemma}{\ref{norme-difference}}
Let $x_1, \cdots , x_p , y_1 ,\cdots , y_p \in E$ such that $\forall i, ||x_i|| = ||y_i|| = 1$. Then 
\begin{enumerate}[(i)]
\item $||x_1 \wedge \cdots \wedge x_p - y_1 \wedge \cdots \wedge y_p ||_{\wedge1,p} \leq p.p!.\max (||x_i - y_i||)$,
\item $||x_1 \wedge \cdots \wedge x_p - y_1 \wedge \cdots \wedge y_p ||_{\wedge2,p} \leq p.\max (||x_i - y_i||)$.
\end{enumerate}
\end{customlemma}

\begin{proof}

For $k \in \{1,\dots,p\}$, we set $\alpha_k=x_1 \wedge \dots \wedge x_k$ and $\beta_k= y_1 \wedge \dots \wedge y_k$. Note that $\alpha_k=\alpha_{k-1}\wedge x_k$ and $\beta_k= \beta_{k-1} \wedge y_k$.

Let $d= \max(||x_i - y_i||)$. Then
\[
\begin{array}{llr}
||\alpha_p - \beta_p ||_{\wedge1,p} &\leq || \alpha_{p-1} \wedge (x_p - y_p) ||_{\wedge1,p} + || (\alpha_{p-1}-\beta_{p-1}) \wedge y_p||_{\wedge1,p} \\
&\leq p.d.||\alpha_{p-1}||_{\wedge1,p-1} + p.||\alpha_{p-1} - \beta_{p-1}||_{\wedge1,p-1} \\
&\leq  d.p! + p.||\alpha_{p-1} - \beta_{p-1}||_{\wedge1,p-1}
\end{array}\]
and the desired result follows from an easy induction.
A similar reasoning gives the result for $\|.\|_{\wedge 2,p}$.
\end{proof}

\begin{customlemma}{\ref{normetenseur}}
Let $M \in \mathcal{L}(E)$ and consider the operator norm of $\widehat{M}$ on $\bigwedge\nolimits^pE$ given by \linebreak $\|\widehat{M}\|_{\wedge,p}:=\underset{\|x\|_{\wedge,p}=1}{\sup} \|\widehat{M}(x) \|_{\wedge,p}$. Then this supremum is given by its value on the set of decomposable tensors, i.e.
$$\|\widehat{M}\|_{\wedge,p}=\underset{\|x_1\wedge \cdots \wedge x_p\|_{\wedge,p}=1}{\sup} \|\widehat{M}(x_1\wedge \cdots \wedge x_p) \|_{\wedge,p}.$$

%For $\|.\|_{\wedge,pr}$, $\|\widehat{M}\|\leq \sqrt{p!} \|\widehat{M}\|_0$.
\end{customlemma}

\begin{proof} Let $\|\widehat{M}\|_0:=\underset{\|x_1\wedge \cdots \wedge x_p\|_{\wedge,p}=1}{\sup} \|\widehat{M}(x_1\wedge \cdots \wedge x_p) \|_{\wedge,p}$. By definition of $\|\widehat{M}\|$ we have $\|\widehat{M}\|_{\wedge,p} \geq \|\widehat{M}\|_0$. Let $u \in \bigwedge\nolimits^pE$ and write $u=\sum x_1^i \wedge \cdots \wedge x_p^i$ a decomposition of $u$.
Then $$\|\widehat{M}u\| \leq \sum \|\widehat{M}(x_1^i \wedge \cdots \wedge x_p^i)\| \leq \|\widehat{M}\|_0 \sum \|x_1^i \wedge \cdots \wedge x_p^i\|\leq \|\widehat{M}\|_0 \sum \|x_1^i\| \dots \| x_p^i\|.$$ Taking the infimum over all decompositions of $u$, we get $\|\widehat{M}u\| \leq \|\widehat{M}\|_0. \|u\|$ hence $\|\widehat{M}\| \leq \|\widehat{M}\|_0$.

%\emph{Case 2: $\|.\|=\|.\|_{\wedge,pr}$.}
%Let $u \in \bigwedge\nolimits^pE$. Then there exists an orthonormal family $(e_1,\cdots,e_k)$ such that we can write $u=\underset{I}\alpha_I e_I$, where $I =\{i_1,\cdots,i_p\}$ with $1\leq i_1 < \cdots < i_p \leq k$ and $e_I=e_{i_1} \wedge \cdots \wedge e_{i_p}$. By considering a bigger family $(e_1,\cdots, e_l)$, we can assume : $\forall i \in \{1,\cdots,k\}, Me_i \in \Span(e_1,\cdots, e_l)$, i.e. $Me_i = \underset{j}{\sum}a_{i,j} e_j$. Then
%$\widehat{M}e_I=(\underset{j}{\sum}a_{i_1,j}e_j) \wedge \cdots \wedge (\underset{j}{\sum}a_{i_p,j}e_j) = \underset{J}{\sum} \Delta_{I,J} e_J$, where $\Delta_{I,J}$ is the determinant of the matrix $(a_{i,j})_{i \in I,j \in J}$. This gives $\widehat{M}u=\underset{I}{\sum} \alpha_I \widehat{M}e_I=\underset{I}{\sum} \underset{J}{\sum} \alpha_I \Delta_{I,J} e_J = \underset{J}{\sum} (\underset{I}{\sum} \alpha_I \Delta_{I,J})e_J$. We now get, using the Cauchy-Schwarz inequality,
%$$
%\begin{array}{rll}
% \|\widehat{M}u\|^2 &= \underset{J}{\sum} |\underset{I}{\sum} \alpha_I \Delta_{I,J}|^2 \\
% &\leq \underset{J}{\sum} (\underset{I}{\sum} |\alpha_I|^2 \underset{I'}{\sum}|\Delta_{I',J}|^2) \\
% &=\underset{I}{\sum} |\alpha_I|^2.\underset{I'}{\sum}\underset{J}{\sum}|\Delta_{I',J}|^2  \\
% &= \underset{I}{\sum}|\alpha_I|^2.\|\widehat{M}e_I\|^2 \\
% &\leq \underset{I}{\sum}|\alpha_I|^2.\|\widehat{M}\|_0^2  NUUUL\\
% &\leq \|\widehat{M}\|_0^2\|u\|^2.
%\end{array}$$
%Hence $\|\widehat{M}\| \leq \|\widehat{M}\|_0$.
\end{proof}

\begin{customprop}{\ref{normeproduit}}
Let $M \in \mathcal{L}(E)$. We have
$$\|\widehat{M}\|_{\wedge,p} \leq \|\bigwedge\nolimits^{p-1}M\|_{\wedge,p-1} .\|M\| \leq \|M\|^p$$

Similarly, for a continuous linear map $m : E \rightarrow \mathbb{C}^p$, we have $\|\hat{m}\| \leq \|m\|^p$.
\end{customprop}

\begin{proof}Let $u \in \bigwedge\nolimits^pE$ and write $u=\sum x_1^i \wedge \cdots \wedge x_p^i$ a decomposition of $u$. Then $\widehat{M}u = \sum Mx_1^i \wedge \cdots \wedge Mx_p^i$ is a decomposition of $Mu$ and 
$$\|\widehat{M}u\|_{\wedge,p} \leq \sum \|Mx_1^i\| \dots \|\bigwedge\nolimits^{p-1}M(x_2^i \wedge \cdots \wedge x_p^i\| \leq \sum \|M\| \|\bigwedge\nolimits^{p-1}M\| \|x_1^i\| \dots \|x_p^i\|.$$
Taking the infimum over all decompositions of $u$, we get $\|\widehat{M}u\|_{\wedge,p} \leq \|M\| \|\bigwedge\nolimits^{p-1}M\| \|u\|_{\wedge,p}$, hence $\|\widehat{M}\|_{\wedge,p} \leq \|M\| \|\bigwedge\nolimits^{p-1}M\|$. An easy induction gives the second inequality.

For $m:E \rightarrow \mathbb{C}^p$, a similar computation gives 
$$\begin{array}{rcl}

|\widehat{m}u| &\leq &\sum |\hat{m} (x_1^i \wedge \cdots \wedge x_p^i)|\\
 &\leq& \sum | \det( m(x_1^i), \dots , m(x_p^i))|\\
 &\leq& \sum \| m(x_1^i)\| \cdots \|m(x_p^i)\|\\
 &\leq& \|m\|^p \sum \| x_1^i\| \cdots \|x_p^i\| \end{array}$$
 which gives the desired result.

\end{proof}

\begin{customlemma}{\ref{derivee chapeau}}
For $t \in \mathbb{D}$,
$$\left\|\dfrac{d}{dt} \widehat{M}_\omega(t) \right\| \leq p \left\| \bigwedge\nolimits^{p-1}M_\omega(t) \right\| \left\|\dfrac{d}{dt} M_\omega(t) \right\|$$
Under the assumptions of Theorem \ref{theorem produit}, this gives
$$\left\|\dfrac{d}{dt} \widehat{M}_\omega(t) \right\| \leq p C \left\|\widehat{M}_\omega(t) \right\|$$
where $C = \sup \left\{ \dfrac{\|\bigwedge\nolimits^{p-1}M_\omega(t)\|}{\|\bigwedge\nolimits^{p}M_\omega(t)\|}\|\frac{d}{dt}M_\omega(t)\| : \omega \in \Omega, t \in \mathbb{D}\right\}$.
\end{customlemma}

\begin{proof}
Let $x_1, \dots,x_p \in E$.
$$\begin{array}{rcl}
\dfrac{d}{dt} \widehat{M}_\omega(t) x_1 \wedge \cdots \wedge x_p &= &\dfrac{d}{dt} (M_\omega(t) x_1 \wedge \cdots \wedge M_\omega(t)x_p \\
&= & \sum\limits_{k=1}^p M_\omega(t) x_1 \wedge \cdots \wedge \dfrac{d}{dt} M_\omega(t)x_k \wedge \cdots \wedge M_\omega(t)x_p
\end{array}$$
which gives 
$$\|\dfrac{d}{dt} \widehat{M}_\omega(t) x_1 \wedge \cdots \wedge x_p\| \leq p \|\bigwedge\nolimits^{\!p-1}M_\omega(t)\|.\|\dfrac{d}{dt} M_\omega(t)\| .\|x_1\| \dots\|x_p\|.$$

Let $u \in \bigwedge\nolimits^pE$ and write $u=\sum x_1^i \wedge \cdots \wedge x_p^i$ a decomposition of $u$. Then
$$\|\dfrac{d}{dt} \widehat{M}_\omega(t) u\| \leq p \|\bigwedge\nolimits^{\!p-1}M_\omega(t)\|.\|\dfrac{d}{dt} M_\omega(t)\| . \sum\\|x_1^i\| \dots\|x_p^i\|.$$
Taking the infimum over all decompositions of $u$, we get $\|\dfrac{d}{dt} \widehat{M}_\omega(t) u\| \leq p \|\bigwedge\nolimits^{p-1}M_\omega(t)\|.\|\dfrac{d}{dt} M_\omega(t)\|.\|u\|$, hence $\|\dfrac{d}{dt} \widehat{M}_\omega(t) u\| \leq p \|\bigwedge\nolimits^{p-1}M_\omega(t)\|.\|\dfrac{d}{dt} M_\omega(t)\|$.

\end{proof}


\begin{thebibliography}{BGS}{\small
\addcontentsline{toc}{section}{Bibliographie}


\bibitem[Bir57]{Bir57} 
G.~Birkhoff, {\it Extensions of Jentzsch’s theorem},
Trans. Amer. Math. Soc. 85, 219-227 (1957).


\bibitem[Blu16]{Blu16} 
A.~Blumenthal, {\it A volume-based approach to the multiplicative ergodic theorem on Banach spaces},  Discrete and Continuous Dynamical Systems A 36 (5), 2377-2403 (2016).

\bibitem[BM15]{BM15} 
A.~Blumenthal, I.~Morris, {\it Characterization of dominated splittings for operator cocycles acting on Banach spaces}, submitted.

\bibitem[Dub09]{Dub09} 
L.~Dubois, {\it Projective metrics and contraction principles for complex cones}, Journal of the London Mathematical Society 79, 719–737 (2009).

\bibitem[FGQ15]{FGQ15} 
F.~Froyland,C.~Gonzalez-Tokman, A.~Quas, {\it Stochastic stability of Lyapunov exponents and Oseledets splittings for semi-invertible matrix cocycles}, Communications on Pure and Applied Mathematics 68, 2052-2081 (2015).

\bibitem[Fro08]{Fro08} 
G.~Frobenius, {\it Über Matrizen aus nicht negativen Elementen},
S.-B. Preuss. Akad. Wiss., 456-477 (1908 and 1912).

\bibitem[GTQ15]{GTQ15} 
C.~Gonzalez-Tokman, A.~Quas, {\it A concise proof of the Multiplicative Ergodic Theorem on Banach Spaces}, Journal of Modern Dynamics 9, 237-255 (2015).

\bibitem[Per07]{Per07} 
O.~Perron, {\it Zur Theorie der Matrices},
Math. Ann. 64, 248-263 (1907).

\bibitem[Rue79]{Rue79} 
D.~Ruelle, {\it  Analycity properties of the characteristic exponents of random matrix products}, Advances in Mathematics 32, 68-80 (1979).

\bibitem[Rug10]{Rug10} 
H.H.~Rugh, {\it Cones and gauges in complex spaces: Spectral gaps and complex Perron-Frobenius theory}, Annals of Mathematics 171, 1707–1752 (2010).

}
\end{thebibliography}
\end{document}